\RequirePackage[l2tabu,orthodox]{nag}
\documentclass[a4paper,10pt]{amsart}

\author{Benjamin McKay}
\title{Complex homogeneous surfaces}
\date{\today}

\address{School of Mathematical Sciences, University College Cork, Cork, Ireland}
\email{b.mckay@ucc.ie}

\thanks{This publication has emanated from activity conducted with the financial support of Science Foundation Ireland under the International Strategic Cooperation Award Grant Number SFI/13/ISCA/2844.}

\usepackage{etex}
\usepackage[T1]{fontenc}
\usepackage[utf8]{inputenx}
\usepackage{lmodern}
\usepackage[kerning=true,tracking=true]{microtype}
\usepackage{fixltx2e}
\usepackage[pdftex,pagebackref]{hyperref}
\hypersetup{
  colorlinks   = true,  
  urlcolor     = black, 
  linkcolor    = black, 
  citecolor    = black  
}

\usepackage[usenames,dvipsnames,svgnames,table]{xcolor}
\usepackage{xstring}
\usepackage{array}
\usepackage{amsfonts}
\usepackage{euscript}
\usepackage{amsmath}
\usepackage{amscd}
\usepackage{amssymb}
\usepackage{subfigure}
\usepackage{euscript}
\usepackage{amsthm}
\usepackage{graphicx}
\usepackage{lscape}
\usepackage{booktabs}
\usepackage{pdftexcmds}
\usepackage{tikz-cd}
\usetikzlibrary{decorations.markings}
\usepackage[draft]{varioref}
\usepackage{tikz}
\usepackage{graphics}
\usepackage{xspace}
\usepackage{amsxtra}
\usepackage{braket}
\usepackage{verbatim}
\usepackage{mathrsfs}
\usepackage{longtable}
\usepackage{ragged2e}
\usepackage{enumitem}

\newtheorem{theorem}{Theorem}[section]
\newtheorem{lemma}[theorem]{Lemma}

\newtheorem{proposition}[theorem]{Proposition}
\theoremstyle{remark}
{%
    \newtheorem{example}[theorem]{Example}
}%

\newtheorem{remark}[theorem]{Remark}

\usepackage{makerobust}

\newcommand*{\ii}{\ensuremath{i}}
\newcommand*{\R}[1]{\ensuremath{\mathbb{R}^{#1}}}
\newcommand*{\C}[1]{\ensuremath{\mathbb{C}^{#1}}}
\DeclareRobustCommand*{\Z}[1]{\ensuremath{\mathbb{Z}^{#1}}}
\newcommand*{\Q}[1]{\ensuremath{\mathbb{Q}^{#1}}}
\MakeRobustCommand\Q
\newcommand*{\pr}[1]{\ensuremath{\left(#1\right)}}
\newcommand*{\br}[1]{\ensuremath{\left\{#1\right\}}}
\newcommand*{\wo}[1]{\ensuremath{-#1}}
\newcommand*{{\bull}}{{\scriptscriptstyle{\bullet}}}
\newcommand*{\of}[1]{\ensuremath{\!\pr{#1}}}

\newcommand*{\Sym}[2]{\ensuremath{\operatorname{Sym}^{#1}\of{#2}}}
\newcommand*{\GL}[1]{\ensuremath{\operatorname{GL}\of{#1}}}

\newcommand*{\SL}[1]{\ensuremath{\operatorname{SL}\of{#1}}}
\newcommand*{\PSL}[1]{\ensuremath{\operatorname{PSL}\of{#1}}}

\newcommand*{\Lm}[2]{\ensuremath{\Lambda^{#1}\pr{#2}}}

\newcommand*{\homology}[2]{\ensuremath{H_{#1}\of{#2}}}
\DeclareMathOperator{\Ad}{Ad}

\newcommand*{\im}{\ensuremath{\operatorname{im}}}
\newcommand*{\Proj}[1]{\ensuremath{\mathbb{P}^{#1}}}

\newcommand*{\homotopygroup}[2]{\ensuremath{\pi_{#1}\of{#2}}}
\newcommand*{\fundamentalgroup}[1]{\homotopygroup{1}{#1}}

\makeatletter
\newcommand*{\pd}[3][1]
{
\frac{%
\partial%
\ifnum\pdf@strcmp{#1}{1}=0\else^#1\fi%
#2}%
{%
\partial
\ifnum\pdf@strcmp{#1}{1}=0\else^#1\fi%
#3
}%
}
\makeatother

\makeatletter
\renewcommand*\env@matrix[1][\arraystretch]{%
  \edef\arraystretch{#1}%
  \hskip -\arraycolsep
  \let\@ifnextchar\new@ifnextchar
  \array{*\c@MaxMatrixCols c}}
\makeatother

\newcommand*{\rank}[1]{\ensuremath{\operatorname{rank} #1}}
\newcommand*{\id}{\ensuremath{\operatorname{id}}}
\newcommand*{\Bihol}[1]{\ensuremath{\operatorname{Bihol}\of{#1}}}
\newcommand*{\Hol}[2]{\ensuremath{\operatorname{Hol}\of{#1,#2}}}
\newcommand*{\OnGroup}{\pr{\GL{2,\C{}}/\Z{}_n} \rtimes \Sym{n}{\C{2}}^*}
\newcommand*{\SOnGroup}{\pr{\SL{2,\C{}}/\Z{}_n} \rtimes \Sym{n}{\C{2}}^*}

\newlength{\transposeHeight}

\newcommand*{\Lie}[1]{\ensuremath{\mathfrak{\lowercase{#1}}}}
\newcommand*{\MakeLie}[1]{\expandafter\def\csname Lie#1\endcsname{\Lie{#1}}}

\def\lst{G,H}
\makeatletter
\@for\i:=\lst\do{\expandafter\MakeLie \i}
\makeatother

\newcommand*{\LieSL}[1]{\ensuremath{\mathfrak{sl}\of{#1}}}

\newcommand*{\Aut}[1]{\ensuremath{\operatorname{Aut} #1}}

\newcommand*{\OO}[1]{%
  \ensuremath{%
    \mathcal{O}%
    \IfStrEq{#1}{0}{}{\of{#1}}
  }%
}%
\newcommand*{\OOp}[2]{
  \ensuremath{
    \mathcal{O}
    \IfStrEq{#1}{0}{}{\of{#1}}
    \IfStrEq{#2}{1}{}{^{\oplus{#1}}}
  }
}

\newcommand*{\AtoBtoCtoD}[7]
{%
\[
\begin{tikzcd}[ampersand replacement=\&]
{#1} \arrow{r}{#2} \& {#3} \arrow{r}{#4} \& {#5} \arrow{r}{#6} \& {#7}
\end{tikzcd}
\]
}%

\makeatletter
\newenvironment{tbl}[4][.]%
{
\ifnum\pdf@strcmp{#1}{.}=\z@%
\ifnum\pdf@strcmp{#4}{}=\z@%
\begin{longtable}{@{}#3@{}}%
\endfirsthead%
\multicolumn{#2}{l}{{
\ldots continued}} \\ %
\endhead%
\multicolumn{#2}{r}{{\ldots}} \\ %
\endfoot\endlastfoot
\else
\begin{longtable}{@{}#3@{}}%
\toprule #4 \\ \midrule%
\endfirsthead
\multicolumn{#2}{l}{{
\ldots continued}} \\ 
\toprule #4 \\ \midrule
\endhead%
\bottomrule%
\multicolumn{#2}{r}{{\ldots}} \\
\endfoot%
\bottomrule%
\endlastfoot%
\fi
\else
\ifnum\pdf@strcmp{#4}{}=\z@%
\begin{longtable}{@{}#3@{}}%
\caption{#1} \\ 
\multicolumn{#2}{l}{{\tablename\ \thetable{}: continued}} \\ %
\endhead%
\multicolumn{#2}{r}{{\ldots}} \\ %
\endfoot%
\else%
\begin{longtable}{@{}#3@{}}%
\caption{#1} \\ 
\toprule #4 \\ \midrule%
\endfirsthead
\multicolumn{#2}{l}{{\tablename\ \thetable{}: continued}} \\ 
\toprule #4 \\ \midrule
\endhead%
\bottomrule%
\multicolumn{#2}{r}%
{{\ldots}} \\
\endfoot%
\bottomrule%
\endlastfoot%
\fi
\fi%
}
{
\end{longtable}%
}
\makeatother

\makeatletter
{
\begin{longtable}{@{}#3@{}}%
\ifnum\pdf@strcmp{#1}{}=\z@\else\caption{#1} \\ \fi%
\ifnum\pdf@strcmp{#4}{}=\z@%
\endfirsthead%
\multicolumn{#2}{l}{{\tablename\ \thetable{}: continued}} \\[0pt]%
\endhead%
\multicolumn{#2}{r}{{\ldots}} \\ %
\endfoot%
\endlastfoot%
\else%
\toprule #4 \\ \midrule%
\endfirsthead
\multicolumn{#2}{l}{{\tablename\ \thetable{}: continued}} \\ 
\toprule #4 \\ \midrule
\endhead%
\bottomrule%
\multicolumn{#2}{r}%
{{\ldots}} \\
\endfoot%
\bottomrule \endlastfoot%
\fi%
}
{
\end{longtable}%
}
\makeatother

\newcolumntype{A}{>{$}l<{$}}
\newcolumntype{V}{>{$}r<{$}}

\newcommand*{\Aff}[1]{\ensuremath{\operatorname{Aff}_{\mathbb{C}}\of{#1}}}
\newcommand*{\uAff}[1]{\ensuremath{\widetilde{\operatorname{Aff}}_{\mathbb{C}}\of{#1}}}

\renewcommand{\arraystretch}{1.2}
\newcommand*{\betweenEntries}{\\}

\newcommand{\twobytwo}[4]{\ensuremath{\pr{\begin{smallmatrix}#1&#2\\#3&#4\end{smallmatrix}}}}
\newcommand{\twobytwob}[4]{\ensuremath{\left[\begin{smallmatrix}#1&#2\\#3&#4\end{smallmatrix}\right]}}

\newcommand*{\rG}[1]{\ensuremath{\mathscr{G}_{#1}}}

\begin{document}

\begin{abstract}
We classify the transitive, effective, holomorphic actions of connected complex Lie groups on complex surfaces.
\end{abstract}

\maketitle
\tableofcontents

\section{Introduction}
\label{section:HomogSurfaces}
This paper classifies faithful transitive holomorphic actions of connected complex Lie groups on complex surfaces; see table~\vref{table:complexHomogeneousSurfaces}.
Lie classified the germs near a generic point of holomorphic actions of complex Lie algebras on complex surfaces \cite{Lie:GA:5} p .767--773 (also see \cite{Mostow:1950}, \cite{Olver:1995} p. 472); there are 27 connected families of equivalence classes of actions. 
The intransitive actions are not considered in this paper, but they appear in Lie's classification.
Tits \cite{Tits:1962} classified the compact complex manifolds of dimensions 2 and 3 acted on holomorphically and transitively by complex Lie groups.
Erdman-Snow \cite{Erdman-Snow:1985, Snow:1979} classified the complex solv-manifolds (i.e. quotients of a complex solvable Lie group by a discrete subgroup) of complex dimension 1, 2 or 3. 
Huckleberry and Livorni \cite{Huckleberry:1986,Huckleberry/Livorni:1981} classified the complex surfaces which admit transitive holomorphic actions (also see \cite{Oeljeklaus/Richthofer:1984}).
Winkelmann \cite{Winkelmann:1986,Winkelmann:1988,Winkelmann:1988b,Winkelmann:1989} classified the complex 3-manifolds which admit transitive holomorphic actions.
Those papers classified the complex manifolds but not the actions.
There are also some disconnected complex Lie groups containing the groups listed in table~\vref{table:complexHomogeneousSurfaces}; we ignore these, but the interested reader can consult Huckleberry \cite{Huckleberry:1986} for some information.

\section{Quotients of Lie group actions}

A \emph{complex homogeneous space} is a pair \(\pr{X,G}\) of a connected complex manifold \(X\) with a connected complex Lie group \(G\) acting holomorphically, transitively and effectively on \(X\).
For each transitive Lie algebra action germ in any of Lie's families, there is (up to isomorphism) a unique complex homogeneous surface \(\pr{X,G}\) with \(X\) simply connected inducing this Lie algebra action germ.
In our table, this complex homogeneous surface \(\pr{X,G}\) is listed first, followed by all complex homogeneous surfaces \(\pr{X',G'}\), with \(X'\) not necessarily simply connected but covered by \(X\), with \(G'\) covered by \(G\), with the same Lie algebra action germ.

A \emph{morphism} of complex homogeneous spaces is a pair \(\pr{\delta,h}\) where \(\delta \colon X \to X'\) is a holomorphic map equivariant for a holomorphic group morphism \(h \colon G \to G'\).
The quotient surfaces \(\pr{X',G'}\) (for which \(\delta\) is a submersion) of a given complex homogeneous surface \(\pr{X,G}\) correspond to the discrete groups \(\pi\) acting on \(X\) holomorphically, freely and properly so that every element of \(\pi\) commutes with every element of \(G\): \(X'=X/\pi\), \(G'=G/\pr{\pi \cap G}\).
For example, if \(X=\C{2}\) and \(G=\C{2}\) acts by translation, then any element of \(\pi\) must commute with translations, so must be a translation.
Hence \(\pi\) can be precisely any discrete subgroup of the translation group and \(X'=\C{2}/\pi\).

If a diffeomorphism \(f\) commutes with all elements of \(G\), it preserves the fixed locus of every element of \(G\).
We start our search for candidates \(\pi\) by looking for the biholomorphisms \(f\) which commute with every element of \(G\) and have no fixed points and preserve the fixed point locus of any element of \(G\).
We then ask which discrete groups of such biholomorphisms act freely and properly.

\section{Lie's classification into families}\label{section:Lies.classification}

Lie doesn't specify whether he is classifying the real Lie algebra actions on real surfaces or the complex Lie algebra actions on complex surfaces; the classification turns out to be essentially the same.
Lie categorized the families of holomorphic Lie algebra actions on surfaces according to various series:%
\begin{center}\begin{tabular}{@{}Vl@{}}
A & no invariant holomorphic foliation \\
B & one invariant holomorphic foliation \\
C & 2 invariant holomorphic foliations \\
D & a one parameter family of invariant holomorphic foliations \\
E & an infinite dimensional family of invariant holomorphic foliations. 
\end{tabular}\end{center}
\noindent{}The \(B\) series splits into subseries, according to whether the quotient group on the space of leaves is%
\begin{center}\begin{tabular}{@{}Vl@{}}
B\alpha & not transitive \\
B\beta & 1-dimensional \\
B\gamma & 2-dimensional \\
B\delta & 3-dimensional.%
\end{tabular}\end{center}
\noindent{}In this paper we can ignore the \(B\alpha\) subseries, because the single action it contains is intransitive.
The \(C\) series splits into 9 subseries, with no obvious geometric interpretation, called \(C1, C2, \dots, C9\).
Both \(C1\) and \(C4\) are intransitive, so we ignore them.
The \(D\) series splits into 3 actions, \(D1\), \(D2\) and \(D3\), all of which are transitive. 
The \(E\) series has only one action, which is intransitive, and so we ignore it.

\section{Biholomorphism groups of some bundles}

Recall the biholomorphism groups of the complex homogeneous curves; \(E_a=\C{}/\Z{}[1,a]\) is an arbitrary elliptic curve and \(\omega=e^{\pi i/3}\) and \(\tau\) is not in \(\Z{}[1,i] \cup \Z{}[1,\omega]\).
\begin{center}
\begin{tabular}{AA}
\toprule X & G \\ \midrule
\Proj{1} & \PSL{2,\C{}} \\
\C{} & \Aff{\C{}} = \C{\times} \rtimes \C{} \\
\C{\times} & \C{\times} \sqcup \C{\times} \\ 
E_i & \Set{\pm 1, \pm i} \rtimes E_i  \\
E_{\omega} & \Set{\pm 1, \pm \omega, \pm \omega^2} \rtimes E_{\omega} \\
E_{\tau} & \Set{\pm 1} \rtimes E_{\tau} \\
\bottomrule
\end{tabular}
\end{center}

\begin{example}\label{example:nontrivial.bundle}
Take \(\Lambda \subset \C{}\) a lattice and \(c \in \C{\times}\) a number so that \(c \Lambda = \Lambda\).
Let \(\pi\) be the group of biholomorphisms of \(\C{2}\) generated by 
\(\pr{z,w} \mapsto \pr{z+1,cw}\) and \(\pr{z,w}\mapsto\pr{z,w+\lambda}\) for any \(\lambda \in \Lambda\). 
Let \(S_c=\C{2}/\pi\) and consider the holomorphic bundle \(\C{}/\Lambda \to S_c \to \C{\times}\) given quotienting the maps \(w \mapsto \pr{z_0,w}\) and \(\pr{z,w} \mapsto z\) by \(\pi\)-action.
\end{example}

\begin{proposition}
The biholomorphism group of \(S_c\) is the quotient \(G/\pi\) with \(G\) the group of the following biholomorphisms, where \(b \in \C{\times}\), \(b\Lambda=\Lambda\), \(\lambda_0 \in \Lambda\) and \(f \colon \C{\times} \to \C{}/\Lambda\) is any holomorphic function.
\begin{center}\begin{tabular}{VA}
\toprule
c & \pr{z,w}\mapsto \\
\midrule
 1 & \pr{\pm z + z_0,bw+f\of{e^{2 \pi i z}}} \\
-1 & \pr{\pm z + z_0,bw+\frac{\lambda_0}{2} + e^{\pi i z} f\of{e^{2 \pi i z}}} \\
e^{2 \pi i p/q}\ne \pm 1 & \pr{z + z_0,bw+\frac{\lambda_0}{1-c} + e^{2\pi i pz/q} f\of{e^{2 \pi i z}}} \\
\bottomrule
\end{tabular}\end{center}
\end{proposition}
\begin{proof}
Let \(S=S_c\). 
The fundamental group \(\pi=\fundamentalgroup{S}\) is \(\pi=\Z{} \rtimes \Lambda\).
The fibration \(\pr{z,w}\mapsto z\) descends to a fibration \(S \to \C{\times}\).
The holomorphic functions on \(S\) are precisely the pullbacks via this fibration, because the fibers are compact.
Identify the base of the fibration with the ideals in the algebra of holomorphic functions on \(S\): the fibration is equivariant under biholomorphisms of \(S\).
Pick a biholomorphism of \(S\). 
Composing with the obvious biholomorphism \(\pr{z,w}\mapsto\pr{z+z_0,w}\), arrange that our biholomorphism takes \(z=0\) to \(z=0\).

Suppose that \(c=1\), i.e. the bundle is trivial.
Composing with \(\pr{z,w}\mapsto\pr{-z,w}\), arrange that our biholomorphism is the identity on the base of the fibration, so a 1-parameter family of automorphisms of an elliptic curve, parameterized by maps \(\C{\times} \to \Aut{\C{}/\Lambda}\).

Assume that \(c \ne 1\).
The biholomorphism lifts to a map 
\[
\pr{z,w}\mapsto\pr{Z(z,w),W(z,w)}
\] 
equivariant for an automorphism \(h \colon \pi \to \pi\):
\begin{align*}
Z(z+k, c^k w + \lambda) &= Z(z,w)+k'(k,\lambda), \\
W(z+k, c^k w + \lambda) &= c^{k'(k,\lambda)} W(z,w)+\lambda'(k,\lambda).
\end{align*}
The fibration is preserved so \(Z=Z(z)\) lifts a biholomorphism of \(\C{\times}\): \(Z(z)=\pm z+z_0\) for some \(z_0 \in \C{}\), so \(h(k,\lambda)=\pr{\pm k,\lambda'(k,\lambda)}\). 
The derivatives \(W_z\) and \(W_w\) are invariant under the action of \(\Lambda \subset \pi\), so (for any fixed \(z\)) are functions on the elliptic curve \(\C{}/\Lambda\), so independent of \(w\): \(W(z,w)=a(z)+b w\) for some constant \(b\) with \(b\Lambda=\Lambda\) and some holomorphic function \(a(z)\) so that  \(a(z+k)-c^{\pm k}a(z)+bw\pr{c^k-c^{\pm k}} \in \Lambda\) is constant.

If \(c=-1\) then equivariance says that \(a(z+1)+a(z) \in \Lambda\) is constant, say equal to \(\lambda_0\).
Let \(h(z) = a(z)-\frac{\lambda_0}{2}\), and then \(h(z+1)+h(z)=0\), so \(h(z)\) is a function of \(e^{\pi i z}\), say \(h(z)=g\of{e^{\pi iz}}\), for some \(g(s)\) holomorphic on \(\C{\times}\) and then \(g(s)+g(-s)=0\), i.e. \(g(s)=sf\pr{s^2}\) where \(f\) is holomorphic on \(\C{\times}\).

If \(c \ne 1\) and \(c \ne -1\) then we have \(c^k=c^{\pm k}\) so that \(\pm = +\) and \(Z(z,w)=z+z_0\) and write \(c=e^{2\pi ip/q}\) for some integers \(p,q\) with \(q \ne 0\). Then \(a(z)=\lambda_0/(1-c)+e^{2\pi ipz/q} f\of{e^{2\pi i z}}\) for an arbitrary holomorphic function \(f\) on the base.
\end{proof}

\section{
Quotient-free actions
}
\label{section:QFA}
The following elementary examples of homogeneous spaces admit no quotients.
\begin{center}%
\begin{tabular}{@{}VAAl@{}}
\toprule
& X & G & Description \\
\midrule
A1 & \Proj{2} & \PSL{3,\C{}} & projective plane \\
A2 & \C{2} & \GL{2,\C{}} \rtimes \C{2} & complex affine plane \\
A3 & \C{2} & \SL{2,\C{}} \rtimes \C{2} & complex special affine plane \\
C3 & \C{2} & \Aff{\C{}} \times \Aff{\C{}} & product of affine lines \\
C6 & \Proj{1} \times \C{} & \PSL{2,\C{}} \times \Aff{\C{}} & projective times affine line \\
C7 & \Proj{1} \times \Proj{1} & \PSL{2,\C{}} \times \PSL{2,\C{}} & product of projective lines \\
D3 & \C{2} & \C{\times} \rtimes \C{2} & rescaling and translation  plane \\
\bottomrule
\end{tabular}
\end{center}
For actions \(A2\), \(A3\), \(C3\) and \(D3\) the group \(G\) contains all translations of the plane \(X=\C{2}\). 
Any biholomorphism commuting with translations is a translation.
Each of these groups contains a rescaling, which does not commute with any nontrivial translation.
Therefore these actions have no quotients.

For each of the actions \(A1\), \(A2\), \(A3\), \(C3\), \(C6\), \(C7\) and \(D3\), each point of the surface \(X\) is the unique fixed point of some element \(g \in G\).
A biholomorphism of \(X\) commuting with every element of \(G\) fixes every element of \(X\): these actions have no quotients.

\section{\texorpdfstring%
{$C2$: translation times affine line}%
{C2: translation times affine line}%
}\label{section:C2}

Take \(X=\C{} \times \C{}\) and \(G=\C{} \times \Aff{\C{}}\) the product of the translation and affine groups in the product action.
The quotients of this action are precisely those of the form \(X'=\pr{\C{}/\Delta} \times \C{}\) and \(G'=\pr{\C{}/\Delta} \times \Aff{\C{}}\) for any discrete subgroup \(\Delta \subset \C{}\).

\section{\texorpdfstring%
{$C5$: projective times translation line}%
{C5: projective times translation line}%
}%

Take \(X=\Proj{1} \times \C{}\), \(G=\PSL{2,\C{}} \times \C{}\) with the product action. 
A biholomorphism \(\pr{z,w} \mapsto \pr{Z,W}\) commuting with all elements of \(G\) commutes with translations of \(w\), so \(Z(z,w)=Z(z,0)\) and \(W(z,w)=W(z,0)+w\).
The map \(z \mapsto W(z,0)\) is a holomorphic function on \(\Proj{1}\) so constant: our biholomorphism belongs to \(G\).
For any point \(z_0 \in \Proj{1}\), there is an element \(g \in G\) whose locus of fixed points is precisely \(\br{z_0} \times \C{}\).
Our biholomorphism, to commute with \(g\), preserves this locus, so \(Z\of{z_0,0}=z_0\) for all \(z_0 \in \Proj{1}\): there is a constant \(c\) so that \(\pr{Z(z,w),W(z,w)}=\pr{z,w+c}\).
The quotients of this group action are \(X'=\Proj{1} \times \pr{\C{}/\Delta}\), and \(G'=\PSL{2,\C{}} \times \pr{\C{}/\Delta}\), \(\Delta \subset \C{}\) an arbitrary discrete subgroup.

\section{%
\texorpdfstring%
{$D1$: the translation plane}%
{D1: the translation plane}%
}%
\label{section:D1}
Take \(X=G=\C{2}\).
As for A2, the quotients of this action are precisely \(X'=G'=X/\pi\), for any discrete subgroup \(\pi \subset \C{2}\).
The automorphism group of \(G\) is \(\Aut{G}=\GL{2,\C{}}\).
If \(\pi\) has rank 1 (i.e. is generated by a single element), arrange \(\pi=\left<\pr{1,0}\right>\) so \(X'=\C{\times} \times \C{}\).
If \(\pi\) has rank 2 and lies in a proper complex subspace of \(\C{2}\), arrange that subspace to be the complex span of \((1,0)\) so \(X'=\pr{\C{}/\pi} \times \C{}\) where \(\pr{\C{}/\pi}\) can be any elliptic curve, and by automorphism arrange  \(\pi=\left<\pr{1,0}, \pr{\tau,0}\right>\) for some \(\tau\) with positive imaginary part, uniquely determined up to \(\SL{2,\Z{}}\)-action.
If \(\pi\) has rank 2, and does not lie in a proper complex subspace of \(\C{2}\), arrange that  \(\pi=\left<\pr{1,0},\pr{0,1}\right>\) so \(X'=\C{\times} \times \C{\times}\).

If \(\pi\) has rank 3, then \(\pi\) has complex span all of \(X=\C{2}\) and the real span \(\R{} \pi\) of \(\pi\) contains a unique 1-dimensional complex subspace \(G_0=\R{}\pi \cap \ii \R{} \pi\).
Let \(\pi_0=\pi \cap G_0\) and let \(X_0\) be the orbit of \(G_0\) through the origin; the trivial fibration \(X_0 \to X \to \bar{X}\) quotients to a fibration \(X'_0 \to X' \to \bar{X}'\), \(X'_0=X_0/G_0\) where \(\bar{X}'\) is the space of orbits of \(G_0\) acting on \(X'\).
Since the action is transitive, \(\bar\pi=\pi/\pi_0\) is a discrete subgroup of \(\C{}\) so has rank 0, 1 or 2, and correspondingly \(\pi_0\) has rank 3, 2 or 1.

If \(\bar\pi\) has rank 2, then \(\pi=\left<\pr{1,0},\pr{\tau,\sigma},\pr{0,1}\right>\) after perhaps applying an automorphism.
If \(\bar\pi\) has rank 1, then \(\pi=\left<\pr{1,0},\pr{\tau,0},\pr{0,1}\right>\) after perhaps applying an automorphism.
If \(\bar\pi\) has rank 0 then \(\pi_0\) is a rank 3 discrete lattice in \(\C{}\), a contradiction.
The quotients of \(\pr{X,G}=\pr{\C{2},\C{2}}\) are as follows; take \(\tau\) any complex number in the upper half plane, \(\sigma \in \C{\times}\) and \(\Lambda=\Z{}[1,\tau]\) and let \(\pi \subset \C{2}\) be the group generated by the listed generators.
\begin{center}
\begin{tabular}{VA}
\toprule
\pi & X' \\
\midrule
\pr{1,0} & \C{\times} \times \C{} \\
\pr{1,0}, \pr{0,1} & \C{\times} \times \C{\times} \\
\pr{1,0}, \pr{\tau,0} & \pr{\C{}/\Lambda} \times \C{} \\
\pr{1,0}, \pr{\tau,0}, \pr{0,1} & \C{\times} \times \pr{\C{}/\Lambda} \\
\pr{1,0},\pr{\tau,\sigma},\pr{0,1} & \C{\times} \to X' \to \C{}/\Lambda \\
\bottomrule
\end{tabular}
\end{center}

The delicate case of \(\pi\) generated by \(\pr{1,0},\pr{\tau,\sigma},\pr{0,1}\) clearly reduces to the previous case when \(\sigma=0\), so as we vary \(\sigma\) we find a family of \(\C{\times}\)-bundles over an elliptic curve, and this family contains a trivial bundle at \(\sigma=0\): each bundle \(X_0' \to X' \to \bar{X}\) in the family is topologically trivial.
As Huckleberry and Livorni \cite{Huckleberry/Livorni:1981} p. 1100 point out, a holomorphic nonconstant function on \(X'\) pulls back to a holomorphic nonconstant function on \(X=\C{2}\) invariant under \(\pi\), which one checks forces \(\sigma=0\).
Hence \(X'_0 \to X' \to \bar{X}'\) is a holomorphically trivial bundle just when \(\sigma=0\).
Huckleberry and Livorni \cite{Huckleberry/Livorni:1981} p. 1100 also prove that every topologically trivial \(\C{\times}\)-bundle over any elliptic curve arises as above as \(X'=\C{2}/\pi\).
It is remarkable that \(G'=X'\) is an abelian complex Lie group.

If \(\pi\) has rank 4, then \(G'=X'\) is a complex torus \(X'=\C{n}/\Lambda\), \(\Lambda=\pi\).
As usual we can assume that
\[
\Lambda=\Z{}\left[\lambda_1,\lambda_2,\lambda_3,\lambda_4\right],
\]
with
\[
\lambda_i = e_i + \tau_{1i} e_1 + \dots + \tau_{4i} e_4
\]
so that \(\det \operatorname{Im} \tau > 0\).

If \(\pi\) has rank greater than 4, then some infinite covering space of \(X'\) is a complex torus, compact, a contradiction.

\section{\texorpdfstring%
{$C8$: affine plane with 1-dimensional stabilizer}%
{C8: affine plane with 1-dimensional stabilizer}%
}\label{section:C8}

Looking for suitable Lie algebras, we see that the connected 1-dimensional complex subgroups \(H \subset \C{\times} \times \C{\times}\) are parameterized by taking any \(p \in \C{2}\wo{0}\) and mapping
\[
\lambda \in \C{} \mapsto \pr{e^{p_1 \lambda},e^{p_2 \lambda}} \in H.
\]
Rescaling \(p\) doesn't change \(H\), so denote \(H\) by \(H_p\) for \(p \in \Proj{1}\).
Embed
\[
\pr{\lambda_1,\lambda_2} \in H_p 
\mapsto 
\begin{pmatrix}
\lambda_1 & 0 \\
0 & \lambda_2
\end{pmatrix}
\in
\GL{2,\C{}}.
\]
Define a group of affine transformations \(G_p = H_p \rtimes \C{2} \subset \Aff{\C{2}}\).
If \(p=0\) or \(p=\infty\), the action of \(G_p\) on \(\C{2}\) is \(C2\); see 
section~\vref{section:C2}.
If \(p=1\), the action of \(G_1\) on \(\C{2}\) is \(D3\); see 
section~\vref{section:QFA}.
For \(p\ne 0, 1, \infty\), Lie denoted the action of \(G_p\) on \(\C{2}\) as \(C8\).
As for A2, \(C8\) has no quotients.

\section{\texorpdfstring%
{$C9$: the affine quadric surface}%
{C9: the affine quadric surface}%
}%
\label{section:C9}

\subsection{Definition}

The surface \(X=\Proj{1} \times \Proj{1} \setminus \operatorname{diagonal}\) is the set of ordered pairs of distinct points of the projective line; \(G=\PSL{2,\C{}}\) acts on the projective line and so on its pairs of points.
Identify \(X\) with a smooth affine quadric surface in \(\pr{b^2-4ac=1} \subset \C{3}\) by
\[
\pr{\alpha, \beta} \in \Proj{1} \times \Proj{1} \setminus \operatorname{diagonal}
\mapsto
\pr{x,y,z}=
\pr{\frac{1}{\alpha-\beta},\frac{\alpha+\beta}{\alpha-\beta},\frac{\alpha\beta}{\alpha-\beta},} \in \C{3}.
\]
View the quadric \(X\) as the set of traceless \(2 \times 2\) matrices \(A \in \LieSL{2,\C{}}\) so that \(\det A=1\), an adjoint orbit of \(G\).
The surface \(X\) is a bundle of affine lines over a projective line,  \(\C{} \to X \to \Proj{1}\), by the fibration taking \((\alpha,\beta) \mapsto \alpha\), with fiber parameterized by \(\frac{\alpha \beta}{\alpha-\beta}\).

\subsection{Quotient spaces}

Given any point of \(X\), say \(\pr{\alpha,\beta}\) with \(\alpha \ne \beta \in \Proj{1}\), the quadratic polynomial \(a \, z^2 + b \, z + c=\pr{\alpha-z}\pr{\beta-z}\)
has roots at \(\alpha\) \and \(\beta\). 
This quadratic polynomial is uniquely determined up to scaling, and satisfies \(b^2-4 \, ac \ne 0\), since it has two distinct roots. 
Map
\[
\pr{\alpha,\beta} \in X \mapsto
\begin{bmatrix}
a \\
b \\
c
\end{bmatrix}
\in X'=\Proj{2} \setminus \pr{b^2=4 \, ac},
\]
so \(X \to X'\) is a 2-1 covering space, and \(X'\) is also a homogeneous space for the same Lie group: \(G'=G\).

Given any two distinct points \(\alpha,\beta\), there is a 1-parameter subgroup of \(G\) fixing precisely \(\alpha\) and \(\beta\) in \(\Proj{1}\), and therefore  acting on \(X\) fixing or interchanging the two points \(\pr{\alpha,\beta}, \pr{\beta,\alpha}\in X\).
Any biholomorphism of \(X\) commuting with all elements of \(G\) leaves invariant the fixed locus of each element of \(G\), so leaves every point of \(X\) fixed or maps \(\pr{\alpha,\beta} \to \pr{\beta,\alpha}\).
The biholomorphisms of \(X\) commuting with \(G\) are \(\id, \pr{\alpha,\beta}\mapsto\pr{\beta,\alpha}\), so \(X'\) is the only quotient.

The biholomorphism group of the affine quadric is an infinite dimensional Banach--Lie 
group \cite{Huckleberry/Isaev:2009}.
The group of regular algebraic morphisms of the affine quadric is not known \cite{Totaro:2007}.
By our global classification of Lie group actions in this paper, neither group contains any finite dimensional transitive complex Lie subgroup other than \(\PSL{2,\C{}}\) and its conjugates.

\section{%
\texorpdfstring%
{$D2$: the affine group}%
{D2: the affine group}%
}%
\label{section:D2}

\subsection{Definition}

Up to biholomorphic isomorphism, there is a unique 2-dimensional connected and simply connected nonabelian complex Lie group; call it \(G=\uAff{\C{}}\), since it is the universal covering group of the group of complex affine transformations of \(\C{}\).
Write elements of \(\uAff{\C{}}\) as pairs \((a,b)\):
\[
\pr{a,b}\pr{a',b'}=\pr{a+a',b+e^a b'}
\]
associating to each pair an affine transformation \(z \mapsto e^a z + b\).
Represent \(\uAff{\C{}}\) as matrices:
\[
\begin{pmatrix}
e^a & 0 & b \\
0 & 1 & a \\
0 & 0 & 1
\end{pmatrix}
\]
with the exact sequence \(1 \to 2 \pi i \Z{} \times \left\{0\right\} \to \uAff{\C{}} \to \Aff{\C{}} \to 1\).
The center of \(\uAff{\C{}}\) is precisely \(2 \pi i \Z{} \times \left\{0\right\}\).
Let \(X=\uAff{\C{}}\) and let \(G=\uAff{\C{}}\) act on \(X\) by left action.

\subsection{Quotient spaces}

\begin{lemma}
The biholomorphic automorphisms of \(G\) are precisely the maps
\[
(a,b) \mapsto \pr{a, \gamma \pr{1-e^a} + \beta b}
\]
for \(\beta \in \C{\times}, \gamma \in \C{}\).
In particular, all biholomorphic automorphisms of \(G\) are inner.
\end{lemma}
\begin{proof}
Suppose that \(\pr{A,B}=\pr{A(a,b),B(a,b)}\) is a holomorphic group morphism, so
\begin{align*}
A\of{a+a',b+e^ab'}&=A\of{a,b}+A\of{a',b'}, \\
B\of{a+a',b+e^ab'}&=B\of{a,b}+e^{A\of{a,b}} B\of{a',b'},
\end{align*}
for all \(a,b,a',b' \in \C{}\).
In particular, \(A\of{a+a',0}=A\of{a,0}+A\of{a',0}\), so that \(A\of{a,0}=\alpha a\) for some constant \(\alpha \in \C{}\).
Similarly, \(A\of{0,b+b'}=A\of{0,b}+A\of{0,b'}\) for all \(b, b' \in \C{}\) so \(A\of{0,b}=\mu b\) for some constant \(\mu\).
But then
\begin{align*}
A\of{a,e^a b'} 
&= 
A\of{a,0}+A\of{0,b'},\\
&=
\alpha a + \mu b',
\\
&=
A\of{0,b'}+A\of{a,0}, \\
&=
A\of{a,b'},
\\
&=
A\of{a,0}+A\of{0,e^{-a}b'},
\\
&=
\alpha a + \mu e^{-a} b',
\end{align*}
for any \(a, b'\), so that \(\mu=0\) and \(A\of{a,b}=\alpha a\) for every \(a,b \in \C{}\).
So \(B\of{0,b+b'}=B\of{0,b}+B\of{0,b'}\) and \(B\of{0,b}=\beta b\) for some \(\beta \in \C{}\).
Taking the rule
\[
B\of{a+a',b+e^ab'}=B\of{a,b}+e^{A\of{a,b}} B\of{a',b'},
\]
and first taking \(b=a'=0\),
\[
B\of{a,e^a b'}=B\of{a,0}+e^{A\of{a,0}}B\of{0,b'},
\]
and then letting \(b=e^a b'\), we find
\begin{align*}
B\of{a,b}&=B\of{a,0}+e^{\alpha a}B\of{0,e^{-a} b},
\\
&=B\of{a,0}+\beta e^{\pr{\alpha-1}a} b.
\end{align*}
On the other hand, taking \(a=b'=0\),
\[
B\of{a',b}=B\of{0,b}+B\of{a',0},
\]
so that
\begin{align*}
B\of{a,b}&=B\of{0,b}+B\of{a,0},
\\
&=B\of{a,0}+\beta b.
\end{align*}
So therefore \(\beta=0\) or \(\alpha=1\).
If \(\beta=0\), then \(\pr{A,B}\) is not a biholomorphism.
So \(\alpha=1\) and \(B\of{a,b}=B\of{a,0}+\beta b\), and the equation 
\[
B\of{a+a',b+e^ab'}=B\of{a,b}+e^{A\of{a,b}} B\of{a',b'}
\]
says that (if we let \(B(a)=B\of{a,0}\)):
\[
B\of{a+a'}=B\of{a}+e^a B\of{a'}=B\of{a'}+e^{a'} B\of{a}.
\]
Let \(a'=h\) and 
\[
\frac{B\of{a+h}-B\of{a}}{h} = e^a \frac{B\of{h}}{h},
\]
so that \(B'(a)=e^a B'(0)\), and we integrate to find \(B(a)=\gamma \pr{1-e^a}\) for some constant \(\gamma\):
\[
A(a,b)=a, B(a,b)=\gamma \pr{1-e^a}+ \beta b,
\]
which is an isomorphism.
\end{proof}

\begin{lemma}
Suppose that \(G\) is a connected nonabelian complex Lie group of 2 complex dimensions and \(X=G/\pi\) where \(\pi \subset G\) is any discrete central subgroup.
No compact complex surface admits any holomorphic \(\pr{X,G}\)-structure.
As a special case, \(G\) contains no cocompact lattice. 
In particular, the complex affine group \(\Aff{\C{}}\) contains no cocompact lattice, nor does its universal covering group \(\uAff{\C{}}\).
\end{lemma}
\begin{remark}
More generally, the same argument holds for \(\pr{X,G}\)-structures, with \(X=G/H\),
whenever there is a form \(\xi \in \Lm{n-1}{\LieG/\LieH}^H\) (where \(n=\dim \LieG - \dim \LieH\)) which is not closed in the differential of Lie algebra cohomology.
\end{remark}
\begin{proof}
Suppose that \(\pi \subset G\) is a cocompact lattice. 
Then \(M=G/\pi\) has a holomorphic \(\pr{X,G}\)-structure. 
More generally suppose that \(M\) is a compact complex surface with flat holomorphic \(\pr{X,G}\)-connection.
Write the induced Cartan connection \cite{Sharpe:1997} (induced by the right invariant Maurer--Cartan 1-form
on the model \(X\)) as 
\[
\omega=
\begin{pmatrix}
\alpha & \beta \\
0 & 0
\end{pmatrix}.
\]
By the Maurer--Cartan equation, \(0=d \alpha = d \beta -\alpha \wedge \beta\).
Since 
\[
\alpha \wedge \beta \wedge \bar\alpha \wedge \bar\beta
\]
is a volume form,
\begin{align*}
0 
&<
\int_M \alpha \wedge \beta \wedge \bar \alpha \wedge \bar \beta,
\\
&=
\int_M d\beta \wedge d \bar\beta,
\\
&=
\int_M d \pr{\beta \wedge d \bar\beta}
,
\\
\intertext{to which we apply Stokes's theorem:}
&=
\int_{\partial  M} \beta \wedge d \bar\beta,
\\
&=
0.
\end{align*}
\end{proof}

\begin{remark}
If \(X=G\) is a group, then the transformations of \(X\) which commute with \(G\) are maps \(f \colon G \to G\) so that \(f(gh)=g\, f(h)\) for all \(g, h \in G\), i.e. \(f(g)=g \, f(1)\), i.e. right translations. 
Any quotient \(X'\) is a quotient by a group \(\pi \subset G\) of right translations acting properly.
The quotient group \(G'\) is \(G/\pr{G \cap \pi}\), but where \(\pi\) acts on the right and \(G\) on the left, so \(G \cap \pi\) means \(Z(G) \cap \pi\) as subgroups of \(G\).
For \(X=\uAff{\C{}}\), the center of \(G=\uAff{\C{}}\) is 
\[
Z=Z(G)=\Set{\pr{2 \pi i k,0}|k \in \Z{}}
\]
so our quotients are \(X'=G/\pi\) and \(G'=G/\pr{Z \cap \pi}=G\) or \(G/nZ\), for some integer \(n \ge 0\), for each discrete subgroup \(\pi \subset G\).
\end{remark}

\begin{remark}
The foliation \(da=0\) of \(\uAff{\C{}}\) is biinvariant.
Any quotient \(X'\) inherits a \(G'\)-invariant holomorphic foliation.
Since \(G'\) acts transitively on \(X'\), \(G'\) also acts transitively on the space of leaves of the foliation.
The leaf space is a Riemann surface; call it \(\bar{X}'\).
The \(G\)-equivariant diagram
\[
\begin{tikzcd}
X \arrow{d} \arrow{r} & X' \arrow{d} \\
\bar{X} \arrow{r} & \bar{X}'
\end{tikzcd}
\]
makes \(\bar{X}'\) a quotient of the translation action on \(\bar{X}=\C{}\), so \(\bar{X}'=\C{}, \C{\times}\) or an elliptic curve \(\C{}/\Lambda\).
\end{remark}

\begin{remark}
For any quotient \(\pr{X',G'}\) of \(\pr{X,G}\), treat the discrete group \(\pi\) as a subgroup of \(G\) acting on the right.
Our group \(G=\uAff{\C{}}\) has an exact sequence
\[
1 \to \C{} \to G \to \C{} \to 1
\]
given by inclusion of \(\C{} = \left\{0\right\} \times \C{} \subset G=\C{2}\) and projection \((a,b) \in G \mapsto a \in \C{}\).
Applied to \(\pi\) this yields an exact sequence, say
\[
1 \to \pi_0 \to \pi \to \bar{\pi} \to 1,
\]
giving a holomorphic fiber bundle \(X'_0=\C{}/\pi_0 \to X'=X/\pi \to \bar{X}'=\C{}/\bar\pi\).
In particular, \(\pi_0, \bar\pi \subset \C{}\) are both discrete subgroups.
\end{remark}

\begin{remark}
For any discrete group \(\pi\) of automorphisms of \(X\) commuting with \(G\), the discrete subgroup \(\bar\pi \subset \C{}\) is fixed by automorphisms of \(G\), an invariant of any quotient \(\pr{X',G'}\).
The kernel \(1 \to \pi_0 \subset \pi \to \bar\pi \to 1\) is acted on by choice of the parameter \(\beta \in \C{\times}\).
Arrange that \(\pi_0=\left\{0\right\}\) or \(\pi_0=\Z{}\) or \(\pi_0=\Z{}[1,\tau]\) for some \(\tau\) in the upper half plane.
We will see that the generators of \(\pi\) are as in table~\vref{table:discrete.subgroups.aff.c}, up to isomorphism.
\end{remark}

\begingroup
\newcommand*{\abelian}{}
\newcommand*{\notabelian}{\cellcolor{gray!15}}
\begin{longtable}{@{}VAAA@{}}
\caption{Lists of generators for all of the discrete subgroups \(\pi \subset \uAff{\C{}}\) up to biholomorphic isomorphism of \(\uAff{\C{}}\).
Notation: \(a \in \C{\times}\), \(k \in \Z{\times}\), \(\tau \in \C{}\) with \(\im{\tau}>0\), \(b \in \C{}\), \(a_1, a_2\) are \(\R{}\)-linearly independent, \(k \in \Z{}\).
For the cells marked in gray, the group \(\pi\) generated by the entries is \emph{not} abelian.}\label{table:discrete.subgroups.aff.c} \\
\toprule
 & \multicolumn{3}{c}{\(\pi_0\)} \\ \cmidrule(r){2-4}
\rank{\bar\pi} & \left\{0\right\} & \Z{} & \Z{}[1,\tau] \\
\midrule
\endfirsthead
\multicolumn{3}{l}{continued} \\
\toprule
           & \multicolumn{3}{c}{\(\pi_0\)} \\ \cmidrule(r){2-4}
\rank{\bar\pi} & \left\{0\right\} & \Z{} & \Z{}[1,\tau] \\
\midrule
\endhead
\bottomrule
\multicolumn{3}{r}{\ldots} \\
\endfoot
\bottomrule
\endlastfoot
0 & \abelian\varnothing & \abelian\pr{0,1} & \abelian\pr{0,1}, \pr{0,\tau} \\
\\
1 & \abelian\pr{2\pi i k, 1} &  \abelian\pr{2 \pi i k,b}, \pr{0,1} &  \abelian\pr{2 \pi i k, b}, \pr{0,1}, \pr{0,\tau} 
\\
	& \abelian\pr{a,0} 
	& \notabelian \pr{2 \pi i\pr{k+\frac{1}{2}},0}, \pr{0,1} 
	& \notabelian \pr{2\pi i \pr{k+\frac{1}{2}}, 0}, \pr{0,1}, \pr{0,\tau}\\
&  &  &  \notabelian \pr{\pi i \pr{k+\frac{1}{2}}, 0}, \pr{0,1}, \pr{0,i}\\
&  &  &  \notabelian \pr{2\pi i \pr{k+\frac{1}{6}}, 0}, \pr{0,1}, \pr{0,e^{\pi i/3}}\\
&  &  &  \notabelian \pr{2\pi i \pr{k+\frac{2}{6}}, 0}, \pr{0,1}, \pr{0,e^{\pi i/3}}\\
&  &  &  \notabelian \pr{2\pi i \pr{k+\frac{4}{6}}, 0}, \pr{0,1}, \pr{0,e^{\pi i/3}}\\
&  &  &  \notabelian \pr{2\pi i \pr{k+\frac{5}{6}}, 0}, \pr{0,1}, \pr{0,e^{\pi i/3}}\\
\\
2 & \abelian\pr{a_1,0}, \pr{a_2,0} & 
\end{longtable}
\endgroup

\begin{lemma}
Up to a biholomorphic automorphism of \(\uAff{\C{}}\), every discrete subgroup \(\pi \subset \uAff{\C{}}\) is in table~\vref{table:discrete.subgroups.aff.c}.
\end{lemma}
\begin{proof}
If \(\bar\pi=\left\{0\right\}\) then \(\pi=\pi_0 \subset \C{}\) is \(\pi=\left\{0\right\}\) or \(\pi=\Z{}\) or \(\pi=\Z{}[1,\tau]\) for some \(\tau\) in the upper half plane, up to automorphism of \(G\).

If \(\bar\pi\) is generated by one element, and \(\pi_0=\left\{0\right\}\), then \(\pi=\left<\pr{a,b}\right>\) for some \(\pr{a,b} \ne \pr{0,0}\).
By an automorphism, 
\[
(a,b) \mapsto \pr{a, \gamma \pr{1-e^a} + \beta b}
\]
we can then arrange \(\pr{a,b}=\pr{a,0}\) or \(\pr{2 \pi i k, 1}\) for some \(k \in \Z{\times}\).

If \(\pi_0\) is generated by one element, we can arrange that that element is \(\pr{0,1}\). Take any element \(\pr{a,b} \in \pi\) and calculate the commutator
\[
\left[\pr{a,b},\pr{0,1}\right]
=
\pr{0,e^a-1} \in \pi_0.
\]
So \(e^a \in \Z{}\), and therefore replacing \(\pr{a,b}\) by its inverse, also \(e^{-a} \in \Z{}\) for all \(\pr{a,b} \in \pi\).
Therefore \(e^a = \pm 1\) so \(a = \pi i k\) for some \(k \in \Z{}\).
In particular, \(\bar\pi\) can be generated by at most one element: \(\bar\pi=\left<\pi i k\right>\).
If \(k\) is odd, we can arrange by automorphism that \(\pi\) is generated by \(\pr{\pi i k, 0}, \pr{0,1}\).

If \(\bar\pi\) is generated by two elements and \(\pi_0=\left\{0\right\}\) then \(\pi\) is generated by two elements \(\pr{a_1,b_1}, \pr{a_2, b_2}\) with \(a_1, a_2\) \(\Z{}\)-linearly independent.
Since \(\bar\pi\) is discrete, \(a_1, a_2\) are \(\R{}\)-linearly independent.
Let
\[
A=
\begin{pmatrix}
1-e^{a_1} & b_1 \\
1-e^{a_2} & b_2
\end{pmatrix}.
\]
There is an automorphism which takes \(\pr{b_1,b_2}\mapsto\pr{0,0}\) just when the linear system
\[
A
\begin{pmatrix}
\gamma \\
\beta
\end{pmatrix}
=
\begin{pmatrix}
0 \\
0
\end{pmatrix}
\]
has solutions \(\gamma, \beta\ne 0\).
Take commutator:
\[
\pr{a_1,b_1}
\pr{a_2,b_2}
\pr{a_1,b_1}^{-1}
\pr{a_2,b_2}^{-1}
=
\pr{0,-\det A} \in \pi_0.
\]
Since we assume that \(\pi_0=\left\{0\right\}\), \(\det A=0\); solve with some \(\pr{\gamma,\beta}\ne\pr{0,0}\).
If every solution \(\pr{\gamma,\beta}\) has \(\beta=0\), then \(e^{a_1}=e^{a_2}=1\), but \(a_1, a_2\) are linearly independent over \(\R{}\).
So there is a solution \(\pr{\gamma,\beta}\) with \(\beta \ne 0\) and we can assume that \(b_1=b_2=0\).
So \(\pi \subset \Lambda' \times \left\{0\right\}\) where \(\Lambda'=\left<a_1,a_2\right>\) so that \(\pi \subset G\) is discrete.

If \(\pi_0\) is generated by two elements, we can arrange by automorphism that \(\pi_0\) is generated by \(\pr{0,1}, \pr{0,\tau}\) with the imaginary part of \(\tau\) positive.
Take commutators:
\begin{align*}
\left[\pr{a,b},\pr{0,1}\right]
&=
\pr{0,e^a-1} \in \pi_0, \\
\left[\pr{a,b},\pr{0,\tau}\right]
&=
\pr{0,\pr{e^a-1}\tau} \in \pi_0.
\end{align*}
So \(e^a \in \Lambda\) and \(e^a \tau \in \Lambda\) and therefore \(e^a \Lambda \subset \Lambda\).
By the same argument, \(e^{-a}\Lambda \subset \Lambda\).
Therefore \(e^a \Lambda = \Lambda\).
This forces \(a\) to lie in the imaginary numbers, so that \(\bar\pi\) has rank at most one.
If \(a\) does not have the form \(2 \pi i k\) for some integer \(k\), then we can arrange \(b=0\) by automorphism.
Proceed through the list of lattices and their symmetries in one complex dimension.
\end{proof}

\begin{remark}
The fundamental group of a product of homogeneous Riemann surfaces is abelian, so if \(\pi\) is not abelian, then \(X' \to \bar{X}'\) is not a trivial bundle.
For each of the abelian groups in table~\vref{table:discrete.subgroups.aff.c}, we can map to a product surface \(X'\) as follows, taking \(\Lambda=\Z{}[1,\tau]\) and \(\Lambda'=\Z{}\left[a_1,a_2\right]\), and map \(\pr{a,b} \in X\) to:
\begin{tbl}{3}{AAA}{\pi & X\to X' & X'}
(0,1) 
& \pr{a,\exp\of{2 \pi i e^{-a} b}} 
& \C{} \times \C{\times} 
\\
(0,1), (0,\tau) 
& \pr{a,e^{-a}b + \Lambda} 
& \C{} \times \pr{\C{}/\Lambda} 
\\
(a',0) 
& \pr{e^{2 \pi i a/a'},b} 
& \C{\times} \times \C{} 
\\
\pr{2\pi i k,1}
& \pr{e^{a/k},e^{-a}b-\frac{a}{2\pi i k}} 
& \C{\times} \times \C{} 
\\
\pr{2\pi i k,b'}, \pr{0,1}
& \pr{e^{a/k},\exp\of{2 \pi i e^{-a}b-\frac{ab'}{k}}} 
& \C{\times} \times \C{\times} 
\\
\pr{2\pi i k,b'}, \pr{0,1}, \pr{0,\tau}
& \pr{e^{a/k},e^{-a}b-\frac{ab'}{2 \pi i k} + \Lambda} 
& \C{\times} \times \pr{\C{}/\Lambda} \\
\pr{a_1,0}, \pr{a_2,0}
& \pr{a+\Lambda',b} 
& \pr{\C{}/\Lambda'} \times \C{}
\end{tbl}
Therefore the group \(\pi\) is abelian precisely when the holomorphic fiber bundle \(X'_0 \to X' \to \bar{X}'\) is holomorphically a product.
\end{remark}

\begin{remark}\label{remark:nonabelian.bundle}
The quotient \(X'=X/\pi\) for \(\pi=\left<\pr{2\pi i\pr{k+\frac{1}{2}},0}, \pr{0,1}\right> \subset \uAff{\C{}}\) is precisely the nonabelian \(\C{\times}\)-bundle over \(\C{\times}\); see Huckleberry and Livorni \cite{Huckleberry/Livorni:1981} p. 1100.
Each choice of \(k\) yields a different complex homogeneous surface, i.e. a different \(\uAff{\C{}}\)-action, but on an isomorphic complex surface.
Similarly, the bundles of elliptic curves corresponding to the other nonabelian discrete subgroups \(\pi \subset G\) are nontrivial, and all such bundles arise in this way uniquely modulo the value of \(k\), but different values of \(k\) yield nonisomorphic group actions. 
\end{remark}

\begin{tbl}[We label the different discrete subgroups of the group $D1$,  modulo isomorphism]{2}{AA}{\text{Label} & \text{Generators of $\pi$}}
D2 & \varnothing \\ 
D2_1 & \pr{0,1} \\ 
D2_2 & \pr{0,1}, \pr{0,\tau} \\
D2_3 & \pr{2\pi i k, 1} \\  
D2_4 & \pr{2 \pi i k,b}, \pr{0,1} \\
D2_5 & \pr{2 \pi i k, b}, \pr{0,1}, \pr{0,\tau} \\
D2_6 & \pr{a,0} \\
D2_7 & \pr{2 \pi i\pr{k+\frac{1}{2}},0}, \pr{0,1} \\
D2_8 & \pr{2\pi i \pr{k+\frac{1}{2}}, 0}, \pr{0,1}, \pr{0,\tau}\\
D2_9 & \pr{\pi i \pr{k+\frac{1}{2}}, 0}, \pr{0,1}, \pr{0,i}\\
D2_{10} & \pr{2\pi i \pr{k+\frac{1}{6}}, 0}, \pr{0,1}, \pr{0,e^{\pi i/3}}\\
D2_{11} & \pr{2\pi i \pr{k+\frac{2}{6}}, 0}, \pr{0,1}, \pr{0,e^{\pi i/3}}\\
D2_{12} & \pr{2\pi i \pr{k+\frac{4}{6}}, 0}, \pr{0,1}, \pr{0,e^{\pi i/3}}\\
D2_{13} & \pr{2\pi i \pr{k+\frac{5}{6}}, 0}, \pr{0,1}, \pr{0,e^{\pi i/3}}\\
D2_{14} & \pr{a_1,0}, \pr{a_2,0} 
\end{tbl}

\begin{tbl}
[%
The intersection of each discrete subgroup of $D1$ with the center of $D1$, i.e. the groups $\pi \cap G$ so that $G'=G/\pr{\pi \cap G}$ are the quotient groups acting on the quotients $X'$. Let $a$ represent an arbitrary complex number which is not rational, and $p/q$ be an arbitrary rational number in lowest terms.
Given a lattice generated by $1,\tau$, write $a$ for a complex number not in the rational span of $1,\tau$, and $p/q+r\tau/q$ for a complex number for which $p, q, r$ are integers with no common factor.
More generally, $\Lambda$ is a lattice in the complex plane.
]%
{3}{AAA}%
{\text{Label} & \text{Generators of $\pi$} & \text{Generators of $\pi \cap Z(G)$}}
D2 & \varnothing & (0,0) \\ 
D2_1 & \pr{0,1} & (0,0) \\ 
D2_2 & \pr{0,1}, \pr{0,\tau} & (0,0) \\
D2_3 & \pr{2\pi i k, 1} & (0,0) \\  
D2_4 & \pr{2 \pi i k,p/q}, \pr{0,1} & \pr{2 \pi i kq,0} \\
D2_4 & \pr{2 \pi i k,a}, \pr{0,1} & \pr{0,0} \\
D2_5 & \pr{2 \pi i k, a}, \pr{0,1}, \pr{0,\tau} & (0,0) \\
D2_5 & \pr{2 \pi i k, \frac{p+r\tau}{q}}, \pr{0,1}, \pr{0,\tau} & \pr{2 \pi i kq, 0} \\
D2_6 & \pr{2 \pi i a,0} & \pr{0,0} \\
D2_6 & \pr{2 \pi i \frac{p}{q},0} & \pr{2\pi i p,0}\\
D2_7 & \pr{2 \pi i\pr{k+\frac{1}{2}},0}, \pr{0,1} & \pr{2 \pi i (2k+1), 0} \\
D2_8 & \pr{2\pi i \pr{k+\frac{1}{2}}, 0}, \pr{0,1}, \pr{0,\tau} &  \pr{2 \pi i (2k+1), 0} \\
D2_9 & \pr{\pi i \pr{k+\frac{1}{2}}, 0}, \pr{0,1}, \pr{0,i} &  \pr{2 \pi i (2k+1), 0} \\
D2_{10} & \pr{2\pi i \pr{k+\frac{1}{6}}, 0}, \pr{0,1}, \pr{0,e^{\pi i/3}} &  \pr{2 \pi i (6k+1), 0} \\
D2_{11} & \pr{2\pi i \pr{k+\frac{2}{6}}, 0}, \pr{0,1}, \pr{0,e^{\pi i/3}} & \pr{2 \pi i (2k+1), 0} \\
D2_{12} & \pr{2\pi i \pr{k+\frac{4}{6}}, 0}, \pr{0,1}, \pr{0,e^{\pi i/3}} &  \pr{2 \pi i (2k+1), 0} \\
D2_{13} & \pr{2\pi i \pr{k+\frac{5}{6}}, 0}, \pr{0,1}, \pr{0,e^{\pi i/3}} &  \pr{2 \pi i (2k+1), 0} \\
D2_{14} & \Lambda \times \left\{0\right\} & \pr{\Lambda \cap 2 \pi i \Z{}} \times \left\{0\right\}
\end{tbl}

\section{%
\texorpdfstring%
{$B\beta{1}$: constant coefficient linear ODE}%
{Bbeta1: constant coefficient linear ODE}%
}%

\subsection{Definition}

Pick an effective divisor \(D\) on \(\C{}\) of positive degree. 
In Olver's notation \cite{Olver:1995} p. 472, the degree of \(D\) is written as \(k\).
Let \(p(z)\) be the monic polynomial with zero locus \(D\) (counting multiplicities).
Let \(V_D\) be the set of all holomorphic functions \(f \colon \C{} \to \C{}\) so that 
\[
p\of{\pd{}{z}} f(z) = 0.
\]
Let \(G_D=\C{} \ltimes V_D\) with the group operation
\[
\pr{t_0,f_0(z)}\pr{t_1,f_1(z)}
=
\pr{t_0+t_1,f_0(z)+f_1\of{z-t_0}}.
\] 
So \(\dim_{\C{}} G_D=1+\deg D\).
Let \(G_D\) act on \(\C{2}\) by the faithful group action
\[
\pr{t,f}\pr{z,w}
=
\pr{z+t,w+f\of{z+t}}.
\] 
The stabilizer of the origin of \(\C{2}\) is the subgroup of pairs \((0,f)\) so that \(f(0)=0\).
The Lie algebra of \(G_D\) is spanned by the vector fields \(\partial_z, f(z) \partial_w\) for \(f \in V_D\).

\subsection{Equivalent actions}

Lie writes the action of \(G_D\) on \(\C{2}\) as \(B\beta{1}\) if \(\deg D \ge 2\), and as \(D1\) if \(D=[0]\) and as \(D2\) if \(D=[\lambda]\) with \(\lambda \ne 0\).
The case \(D1\) is obviously the translation plane as in section~\ref{section:D1}; we ignore it.
We will see below that if \(\lambda \ne 0\) then we can arrange by isomorphism that \(\lambda=-1\), i.e. the action \(D2\) has \(D=[-1]\) and \(G_D\) is isomorphic to \(\uAff{\C{}}\) by \(h \colon (t,f) \in G_{[-1]} \mapsto (a,b) \in \uAff{\C{}}\), \(a=t, f(z)=be^{-z}\), and the homogeneous spaces are identified by the biholomorphism \(\delta \colon (z,w) \in \C{2} \mapsto \pr{z,e^z w} \in \C{2}\).
Henceforth ignore \(D1\) and \(D2\), i.e. \(G_D\) for effective divisors \(D\) of degree 1; assume that our divisor \(D\) has degree \(\ge 2\).

Consider three different families of morphisms.
\begin{enumerate}
\item
If \(g \in G_D\), \(\delta(z,w)=g\of{z,w}\) is equivariant under the isomorphism \(h=\Ad(g) \colon G_D \to G_D\). 
\item
Write a divisor \(D\) as
\[
D = 
n_1 \left[\lambda_1\right] + 
n_2 \left[\lambda_2\right] + 
\dots
+ 
n_k \left[\lambda_k\right],
\]
where \(n_j \in \Z{}_{> 0}\) is a multiplicity, and \(\lambda_i \in \C{}\) is a point.
Pick any nonzero complex numbers \(\mu,\nu\) and
let
\[
E = 
n_1 \left[\mu \lambda_1\right] + 
n_2 \left[\mu \lambda_2\right] + 
\dots
+ 
n_k \left[\mu \lambda_k\right].
\]
Then \(\delta \colon \pr{z,w}\mapsto\pr{\mu z,\nu w}\) is equivariant under
\[
h \colon (t,f) \in G_D \mapsto
\pr{\mu t, \nu f\of{\frac{z}{\mu}}} \in G_E.
\]
\item
Pick some \(f_0 \in V_{D+[0]}\). 
Then \(\delta \colon (z,w) \in \C{2} \mapsto \pr{z,w+f_0(z)} \in \C{2}\) is equivariant under
\[
h \colon \pr{t,f} \in G_D \mapsto \pr{t,f(z)+f_0(z)-f_0(z-t)} \in G_D.
\]
\end{enumerate}

\begin{lemma}\label{lemma:aut.G.D}
Pick effective divisors \(D\) and \(E\) on \(\C{}\).
If there is a biholomorphism \(\delta \colon \C{2} \to \C{2}\) equivariant for a morphism \(h \colon G_D \to G_E\), then modulo compositions of the three morphisms above, \(\delta=\id\) and \(E-D\) is effective or zero and \(h\) is the inclusion morphism.
Isomorphism classes of complex homogeneous spaces \(\pr{\C{2},G_D}\) are parameterized by isomorphism classes of effective divisors on \(\C{}\) modulo rescaling: \(D=\sum n_{\lambda} \left[\lambda\right] \mapsto \sum n_{\lambda} \left[\mu \lambda\right]\).
\end{lemma}
\begin{proof}
Suppose that \(\delta \colon \C{2} \to \C{2}\) is a biholomorphism, equivariant for a holomorphic Lie group morphism \(h \colon G_D \to G_E\).
Since the actions of \(G_D\) and \(G_E\) on \(\C{2}\) are effective, \(h\) is injective.
Arrange that \(\delta(0,0)=(0,0)\) by conjugating with the first of our families of morphisms.

Write \(\delta\) as \(\delta(z,w)=(Z,W)\) and \(h(t,f)=(T,F)\).
Since \(h\) is a homomorphism, \(T\of{t_0+t_1,0}=T\of{t_0,0}+T\of{t_1,0}\), so that \(T\of{t,0}=c t\) for some constant \(c \ne 0\).
Conjugate with a morphism from the second of our families above to arrange that \(T(t,0)=t\).
Equivariance says
\begin{align*}
Z(z+t,w+f(z+t))&=Z(z,w)+T(t,f), \\
W(z+t,w+f(z+t))&=W(z,w)+F_{t,f}(Z(z,w)+T(t,f)).
\end{align*}
Let \(f=0\) to get \(Z(z+t,w)=Z(z,w)+t\), i.e. \(Z(z,w)=z+Z(0,w)\).
Let \(Z(w)=Z(0,w)\).
Since \(h\) is a homomorphism,
\[
T(t,0)+T(0,f)=T(t,f(z-t))=T(0,f)+T(t,0)=T(t,f(z)).
\]
Write \(T(t,f)=t+T(f)\) where \(T(f(z))=T(f(z-t))\) for any \(t\).
Write each \(f \in V_D\) as
\[
f(z) = \sum_j e^{\lambda_j z} \sum_{k=0}^{n_j-1} a_{jk} z^k.
\]
Since \(T(f)\) is holomorphic and additive in \(f\), it is linear in the coefficients \(a_{jk}\) of \(f\).
The coefficients of \(f(z-t)\) vary exponentially as we vary \(t\), unless \(\lambda_j=0\), and so \(T(f)\) depends only on the \(\lambda_j=0\) coefficients.

Suppose first that \(\lambda \ne 0\) is a point of positive multiplicity of \(D\), i.e. \(f(z)=ae^{\lambda z} \in V_D\) for all \(a \in \C{}\).
Then
\[
Z\of{z+t,w+ae^{\lambda z}}=z+Z(w)+t,
\]
for all \(t\), and in particular, if we set \(t=0\), the right hand side is independent of the variable \(a\), so \(Z(z,w)=z\) for all \(z,w\) and \(T(t,f)=t\) for all \(t,f\).

If \(f(z)\) is a polynomial then \(f(z-t)\) has coefficients depending on \(t\) by the binomial theorem, and \(T(f(z-t))\), to be constant in \(t\) depends only on the highest order coefficient of \(f\), say \(T\of{z^n}=b\).
Then 
\[
Z(z+t,w+f(z+t))=Z(z,w)+T(t,f)
\]
but \(T(t,f)=t+T(f)\) so 
\[
Z\of{w+a\pr{z+t}^n}=Z(w)+ab
\]
for any complex numbers \(a\) and \(t\).
Vary \(t\) to see that \(Z(w)\) is constant, and plug in \(Z(0)=0\) to find \(Z(w)=0\) for all \(w\) and \(b=0\), i.e. \(T(f)=0\).
Therefore \(Z(z,w)=z\) for all \(z, w\) and \(T(t,f)=t\) for all \(t, f\).

Plug in \(Z(z,w)=z\) and \(T(t,f)=t\) to our equation for \(W\) to find
\begin{equation}\label{equation:W}
W\of{z+t,w+f(z+t)} = W(z,w) + F_{t,f}\of{z+t}.
\end{equation}
Differentiate both sides in \(w\), and use transitivity of the action of \(G_D\) to conclude that
\[
W(z,w) = \nu w+f_0(z),
\]
for some holomorphic function \(f_0(z)\) and some constant \(\nu\).
Because \(\pr{Z,W}\) is a biholomorphism, \(\nu \ne 0\).
Conjugate with one of the isomorphisms above to arrange that \(\nu=1\).
Equation~\ref{equation:W} becomes
\[
F_{t,f}(z) = f_0(z) - f_0(z-t) + f(z).
\]
Suppose that \(p_E(z)\) is the monic polynomial with divisor \(E\).
Since \(p_E\of{\partial_z}F_{t,f}(z)=0\), plug in \(t=0\) and get \(V_D \subset V_E\), i.e. \(E-D\) is effective or zero.
Instead plug in \(f=0\)  with arbitrary \(t\):
\[
p_E\of{\partial_z}\pr{f_0(z)-f_0(z-t)}=0.
\]
Differentiate in \(t\): \(0=p_{E+[0]}\of{\partial_z}f_0(z)\), i.e. \(f_0(z) \in V_{E+[0]}\).
Conversely, take any \(f_0(z) \in V_{E+[0]}\) and see \(\partial_z p_E\of{\partial_z}f_0(z)=0\) i.e. \(p_E\of{\partial_z}f_0(z)\) is constant, independent of \(z\), so
\[
p_E\of{\partial_z}f_0(z)=p_E\of{\partial_z}f_0(z-t)=0.
\]
Conjugate with our third and final isomorphism above to arrange \(f_0=0\).
So \(F_{t,f}(z)=f(z) \in V_E\) for all \(f \in V_D\), i.e. \(h=\id\), and \(\delta=\id\).
\end{proof}

\subsection{Quotient spaces}

There are infinitely many quotient spaces of \(\C{2}\) with the same Lie algebra action germ as \(\pr{X,G}=\pr{\C{2},G_D}\).
In each of our examples, let \(D\) be any effective divisor on \(\C{}\) of degree \(\ge 2\).

\begin{example}[\(B\beta1A\)]\label{example:B.beta.1.A}
Take any discrete subgroup \(\Delta \subset \C{}\) acting on \(X=\C{2}\) by translations on the second coordinate.
We  can assume that \(\Delta=\Z{}\) or \(\Delta=\Z{}[1,\tau]\) for some complex number \(\tau\) in the upper half plane.
The action of \(G_D\) commutes with that of \(\Delta\) so descends to an action on \(X'=\C{} \times \pr{\C{}/\Delta}\).
\begin{enumerate}
\item[($B\beta1A0$)] If \(D\) has degree \(0\) at \(0\), let \(G'=G_D\).
\item[($B\beta1A1$)] If \(D\) has positive degree at \(0\), \(G'=\C{} \rtimes \pr{V_D/\Delta} = G_D/\Delta\).
\end{enumerate}
\end{example}

In our subsequent examples we make use of the notation: for any complex number \(\lambda\), let \(\Z{} \rtimes_{\lambda} \C{}\) be the group with multiplication
\[
\pr{k,s}\pr{k',s'} = \pr{k+k',s+e^{\lambda k}s'}
\]
and action on \(X=\C{2}\): 
\[
\pr{k,s}\pr{z,w}=\pr{z+k,e^{\lambda k}w+s}.
\]

\begin{example}[\(B\beta1B\)]\label{example:B.beta.1.B}
Suppose that 
\[
D = \left[\lambda\right]
+
\sum_j \left[\lambda + 2 \pi \ii k_j\right]
\]
where \(k_j\) are relatively prime positive integers.
Pick a positive integer \(n\).
Get \(G_D\) to act on \(X'=\C{\times} \times \C{}\) by
\[
\pr{t,f}(Z,W)=\pr{Z',W'}
\]
where we write \(f(z)=e^{\lambda z}P\of{e^{2 \pi i z}}\) and \(Z'=e^{2 \pi i t/n} Z\) and \(W'=e^{-\lambda t}W + P\of{\pr{Z'}^n}\).
This action is equivariant for the local biholomorphism
\[
\pr{z,w} \in \C{2}=X \mapsto \pr{Z,W}=\pr{e^{2 \pi iz/n},e^{-\lambda z}w} \in \C{\times} \times \C{}=X',
\]
Let \(\pi=\left<\pr{n,0}\right> \subset \Z{} \rtimes_{\lambda} \C{}\) and let
\[
G'=G_D/\pr{\Z{} \cap \pr{2 \pi i/\lambda n}\Z{}},
\]
and check that \(X'=X/\pi\) so that \(\pr{X',G'}\) is a quotient of \(\pr{X,G}\).
\end{example}

\begin{example}[\(B\beta1C\)]
Pick some positive integers \(d_1 < d_2 < \dots < d_N\).
Let \(G'\) be the set of pairs \(\pr{\tau,P}\) where \(\tau \in \C{}\) and \(P\) is a complex polynomial of the form \(P(Z)=\sum c_j Z^{d_j}\), with multiplication
\[
\pr{\tau_0,P_0}\pr{\tau_1,P_1}
=
\pr{\tau_0+\tau_1,P_0\of{Z}+P_1\of{e^{-2 \pi i \tau_0} Z}}.
\]
Let \(G'\) act on \(X'=\C{\times} \times \C{}\) by
\[
\pr{\tau,P}(Z,W)=\pr{e^{2 \pi i \tau} Z, W+P\of{e^{2 \pi i \tau}Z}-\tau}.
\]
To translate this action into a quotient action, we let \(m=d_1\) and let \(n\) be the greatest common divisor of the \(d_j-d_1\): \(d_1=m, d_j=m+nk_j\) for relatively prime positive integers \(k_1, k_2, \dots, k_{N-1}\).
Let \(\lambda=2 \pi i m/n\) and let \(D=\left[\lambda\right] + \sum \left[\lambda+2\pi i k_j\right]\).
Let \(\pi = \left<\pr{n,1}\right> \subset \Z{} \times_{\lambda} \C{}\).
Write each element \(f \in V_D\) as \(f(z)=P\of{e^{2 \pi i z/n}}\) for a unique polynomial \(P\), identifying \(G'=G_D\) by \(\pr{t,f}\mapsto\pr{t/n,P}\), i.e. let \(G_D\) act on \(X'\) by
\[
\pr{t,f}\pr{Z,W} = \pr{Z',W'},
\]
where \(Z'=e^{2 \pi i t/n}Z\) and \(W'=W-\frac{t}{n}+P\of{Z'}\).
This action is equivariant for the map \(\pr{z,w} \in X=\C{2} \mapsto \pr{Z,W}=\pr{e^{2 \pi i z/n}, w-z/n}\in X'\).
Let \(\pi=\left<\pr{n,1}\right> \subset \Z{} \rtimes_{\lambda} \C{}\) and 
and check that \(X'=X/\pi\) so that \(\pr{X',G'}\) is a quotient of \(\pr{X,G}\).
\end{example}

\begin{example}[\(B\beta1D\)]
Pick a positive integer \(n\), and relatively prime positive integers \(k_1, k_2, \dots, k_N\).
Let \(D=\left[0\right]+\sum_j \left[2 \pi i k_j\right]\).
Let \(G_D\) act on \(X'=\C{\times} \times \C{\times}\) by
\[
\pr{t,f}\pr{Z,W}=\pr{\dot{Z},\dot{W}},
\]
where \(\dot{Z}=e^{2 \pi i t/n}Z\) and \(\dot{W}=We^{2\pi iP\of{\dot{Z}^n}}\) for \(f(z)=P\of{e^{2 \pi i z}}\).
The map
\[
\pr{z,w} \in X=\C{2} \mapsto \pr{Z,W}=\pr{e^{2 \pi i z/n},e^{2 \pi i w}} \in X'
\]
is \(G_D\)-equivariant, and the fundamental group of \(X'\) is
\[
\pi = \left<(n,0),(0,1)\right> \subset G_D,
\]
so \(G'=G_D/\pi\).
\end{example}

\begin{example}[\(B\beta1E\)]
Pick an integer \(m \ne 0\), and a positive integer \(n\), and relatively prime positive integers \(k_1, k_2, \dots, k_N\) and a complex number \(s\).
Let \(\lambda=2\pi i m/n\).
Let \(D=\left[\lambda\right]+\sum_j \left[\lambda+2 \pi k_j\right]\).
Write each \(f \in V_D\) as \(f(z) = e^{2 \pi i m z/n} P\of{e^{2\pi i z}}\).
Let \(G_D\) act on \(X'=\C{\times} \times \C{\times}\) by
\[
\pr{t,f}\pr{Z,W}=\pr{\dot{Z},\dot{W}},
\]
where \(\dot{Z}=e^{2 \pi i t/n}Z\) and \(\dot{W}=We^{2\pi i\pr{-st/n+\dot{Z}^m P\of{\dot{Z}^n}}}\).
The map \(X=\C{2} \to X'\):
\[
\pr{z,w} \mapsto \pr{Z,W}=\pr{e^{2 \pi i z/n},e^{2 \pi i \pr{nw-sz}/n}}
\]
is \(G_D\)-equivariant, and the fundamental group of \(X'\) is
\[
\pi = \left<(n,s),(0,1)\right> \subset G_D,
\]
so \(G'=G_D/\pi\).
\end{example}

\begin{example}[\(B\beta1F\)]
Pick an integer \(m \ne 0\), and a positive integer \(n\), and relatively prime positive integers \(k_1, k_2, \dots, k_N\).
Let \(\lambda=\pi i (2m+1)/n\).
Let \(D=\left[\lambda\right]+\sum_j \left[\lambda+2 \pi k_j\right]\).
The action of \(G'=G_D\) on \(X=\C{2}\) commutes with the action of 
\[
\pi = \left<(n,0),(0,1)\right> \subset \Z{} \rtimes_{\lambda} \C{}.
\]
The group \(\pi\) acts on \(X\) by \(\pr{kn,s} \in \pi, \pr{z,w} \in X \mapsto \pr{z+kn,(-1)^k w+s}\).
The same group \(\pi\) also acts on \(\C{}\) by \(\pr{kn,s}z=z+kn\).
The map \(\pr{z,w} \in X \mapsto z \in \C{}\) is equivariant.
The quotient \(X'=X/\pi\) maps to the quotient \(\C{}/\pi=\C{\times}\) by \(\pr{z,w} \mapsto e^{2 \pi i z/n}\), a holomorphic principal bundle \(\C{\times} \to X' \to \C{\times}\), the unique nontrivial holomorphic principal \(\C{\times}\)-bundle over \(\C{\times}\); see remark~\vref{remark:nonabelian.bundle}.
\end{example}

\begin{example}[\(B\beta1G\)]
Pick a positive integer \(n\), and relatively prime positive integers \(k_1, k_2, \dots, k_N\) and a complex number \(\tau\) in the upper half plane.
Let \(D=\left[0\right]+\sum_j \left[2 \pi i k_j\right]\), let \(\pi_0=\Z{}[1,\tau]\), and let
\[
\pi = \left<(n,0), (0,1), (0,\tau)\right> \subset \Z{} \times \C{}.
\]
The action of \(G'=G_D\) on \(X=\C{2}\) commutes with the action of \(\pi\), with quotient map
\[
\pr{z,w} \in X=\C{2} \mapsto \pr{Z,W}=\pr{e^{2 \pi i z/n},w+\pi_0} \in X'=\C{2}/\pi=\C{\times} \times \pr{\C{}/\pi_0}.
\]
The action is
\[
\pr{t,f}\pr{Z,w+\pi_0}=\pr{\dot{Z},\dot{w}+\pi_0}
\]
where, if \(f(z)=P\of{e^{2 \pi i z}}\) then \(\dot{Z}=e^{2 \pi i t/n}Z\) and \(\dot{w}=w+P\of{\dot{Z}^n}\).
\end{example}

\begin{example}[\(B\beta1H\)]
Pick an integer \(m \ne 0\), and a positive integer \(n\), relatively prime positive integers \(k_1, k_2, \dots, k_N\), a complex number \(s\) and a complex number \(\tau\) in the upper half plane.
Let \(\lambda=2 \pi i m/n\), let \(D=\left[\lambda\right]+\sum_j \left[\lambda+2 \pi i k_j\right]\), let \(\pi_0=\Z{}[1,\tau]\), and let
\[
\pi = \left<(n,s), (0,1), (0,\tau)\right> \subset Q_D \rtimes \C{}.
\]
The action of \(G'=G_D\) on \(X=\C{2}\) commutes with the action of \(\pi\), with quotient map
\[
\pr{z,w} \in X=\C{2} \mapsto \pr{Z,W}=\pr{e^{2 \pi i z/n},w-\frac{sz}{n}+\pi_0} \in X'=\C{2}/\pi=\C{\times} \times \pr{\C{}/\pi_0}.
\]
The action of \(G'\) on \(X'\) is
\[
\pr{t,f}\pr{Z,w+\pi_0}=\pr{\dot{Z},\dot{w}+\pi_0}
\]
where, if \(f(z)=e^{2\pi i m z/n} P\of{e^{2 \pi i z}}\) then \(\dot{Z}=e^{2 \pi i t/n}Z\) and \(\dot{w}=w-st/n+\dot{Z}^m P\of{\dot{Z}^n}\).
\end{example}

\begin{example}[\(B\beta1I\)]
Pick a positive integer \(n\), a complex number \(\tau\) in the upper half plane, and a complex number \(\lambda\) so that \(e^{\lambda n}\ne 1\) and \(e^{\lambda n} \pi_0=\pi_0\) where \(\pi_0=\Z{}\left[1,\tau\right]\).
For generic \(\tau\), this requires that \(e^{\lambda n}=-1\), but for \(\tau=i\) we can also allow \(e^{\lambda n}=\pm i\), while for \(\tau=e^{2 \pi i/3}\) we can also allow \(e^{\lambda n}=\pm e^{\pm 2 \pi i/3}\).
Pick relatively prime positive integers \(k_1, k_2, \dots, k_N\).
Let \(D=\left[\lambda\right]+\sum_j \left[\lambda + 2 \pi i k_j\right]\).
The group
\[
\pi = \left<(n,0), (0,1), (0,\tau)\right> \subset \Z{} \rtimes_{\lambda} \C{}.
\]
acts on \(\C{2}\) by
\begin{align*}
\pr{n,0}\pr{z,w}&=\pr{z+1,e^{\lambda n} w},  \\
\pr{0,1}\pr{z,w}&=\pr{z,w+1},  \\
\pr{0,\tau}\pr{z,w}&=\pr{z,w+\tau}.
\end{align*}
The surface \(X'=\C{2}/\pi\) is a bundle of elliptic curves \(W+\pi_0 \in \C{}/\pi_0\) over \(e^{2 \pi \ii z} \in \C{\times}\), with transition map \(\pr{z,w+\pi_0} \mapsto \pr{z+1,e^{\lambda n}w+\pi_0}\).
This transition map acts nontrivially on cycles in the elliptic curve, so gives us a topologically nontrivial holomorphic fiber bundle of elliptic curves; see example~\vref{example:nontrivial.bundle}.
Every holomorphic fiber bundle \(\C{}/\Lambda \to X' \to \C{\times}\) arises uniquely in this way.
Since \(\pi\) commutes with the action of \(G_D\), this action drops to an action of \(G'=G_D\) on \(X'\).
\end{example}

\begin{proposition}\label{proposition:quotients.B.beta.1}
Take a quotient \(\pr{X',G'}=\pr{X/\pi,G/\pr{G \cap\pi}}\) of \(\pr{X,G}=\pr{\C{2},G_D}\) for an effective divisor \(D\) on \(\C{}\) of degree \(\ge 2\).
Either \(\pi=\left\{1\right\}\) or, up to isomorphism, \(\pr{X',G'}\) is a unique one of the examples \(B\beta{1}A, B\beta{1}B, \dots, B\beta{1}I\). 
In each of those examples, we can replace \(\tau\) by \(\PSL{2,\Z{}}\)-action without altering the isomorphism type.
Otherwise, all parameters in each of the examples \(B\beta{1}A, \dots, B\beta{1}I\) are determined by the isomorphism type.
\end{proposition}

The proof of the above proposition requires several lemmas.

A \emph{quasiperiod} of an effective divisor \(D\) on \(\C{}\) is a complex number \(\varpi \in \C{}\) so that, for every \(f \in V_D\), \(f(0)=0\) just when \(f(\varpi)=0\).
If \(\lambda\) has order more than \(1\) in \(D\), then \(f(z)=e^{\lambda z}z \in V_D\) vanishes only at \(z=0\), so the only quasiperiod is \(\varpi=0\).
Therefore if there is a nonzero quasiperiod of \(D\), then \(D\) has order 0 or 1 at every point of \(\C{}\).
By translation invariance of \(V_D\), the quasiperiods are a subgroup of \(\C{}\).
Let \(Q_D \subset \C{}\) be the abelian subgroup of quasiperiods of \(D\). 
We will see that each quasiperiod \(\varpi \in Q_D\) determines a complex number, its \emph{weight} \(\gamma_{\varpi} \in \C{\times}\),
by
\[
\gamma_{\varpi} = \frac{f\of{\varpi}}{f(0)}
\]
for all \(f \in V_D\) with \(f(0)\ne 0\).
Construct a group \(Q_D \rtimes \C{}\) with the multiplication
\[
\pr{\varpi_0,s_0}
\pr{\varpi_1,s_1}
=
\pr{\varpi_0+\varpi_1,s_0+\gamma_{\varpi_0} s_1}.
\]
The group \(Q_D \rtimes \C{}\) acts on \(\C{2}\) by
\[
\pr{\varpi,s}(z,w)=\pr{\varpi+z,\gamma_{\varpi} w + s},
\]
effectively and freely.
We will see that \(Q_D \rtimes \C{}\) is the group of biholomorphisms of \(\C{2}\) which commute with the action of \(G_D\).
If \(\pi \subset Q_D \rtimes \C{}\) is a discrete subgroup acting freely and properly on  \(X=\C{2}\), then form the quotient \(X'=\C{2}/\pi\) with quotient action of \(G'=G_D/(G_D \cap \pi)\).

\begin{lemma}\label{lemma:Q.D}
Suppose that \(D\) is an effective divisor on \(\C{}\).
After perhaps replacing \(D\) by a rescaling \(\sum n_j \left[\lambda_j\right] \mapsto \sum n_j \left[\mu \lambda_j\right]\) for some \(\mu \in \C{\times}\), either
\begin{enumerate}
\item
\(Q_D=\left\{0\right\}\) or
\item
\(D=\left[\lambda\right]\) and \(Q_D=\C{}\) or
\item
for some \(\lambda \in \C{}\) and relatively prime positive integers \(k_1, k_2, \dots, k_n\),
\[
D = [\lambda] + \sum_j \left[\lambda + 2 \pi \ii k_j\right]
\]
and \(Q_D = \Z{}\) and every \(f \in V_D\) satisfies \(f(z+1)=e^{\lambda} f(z)\).
\end{enumerate}
\end{lemma}
\begin{proof}
Suppose that \(Q_D \ne \left\{0\right\}\).
Because \(V_D\) is invariant under translation, for any \(z \in \C{}\) and \(f \in V_D\), \(f(z+\varpi)=0\) just when \(f(z)=0\).
If \(D\) has degree 1, i.e. \(D=\left[\lambda\right]\), then every \(f \in V_D\) has the form \(f(z)= c \, e^{\lambda z}\), so \(f(z)=0\) for some \(z\) just when \(c=0\), so just when \(f(z)=0\) for all \(z\), so \(Q_D=\C{}\).

If \(D\) has degree \(n\ge 2\) at some point \(\lambda \in \C{}\), then \(V_D\) contains the function \(f(z)=e^{\lambda z}z\) so there are no quasiperiods.
Therefore quasiperiods exist only when \(D\) has multiplicity zero or one at any point of \(\C{}\), say \(D=\sum_j \left[\lambda_j\right]\), with distinct \(\lambda_1, \lambda_2, \dots, \lambda_n\), so that the functions \(f \in V_D\) are the functions 
\[
f(z) = \sum_j c_j e^{\lambda_j z}.
\]
In particular, if we choose all but two of the coefficients to vanish, say
\[
f(z) = c_1 e^{\lambda_1 z} + c_2 e^{\lambda_2 z} \in V_D,
\]
then the roots of \(f(z)\) lie at the points
\[
z = \frac{\log\pr{-\frac{c_2}{c_1}}}{\lambda_1 - \lambda_2}.
\]
The quasiperiods lie among the differences between the roots:
\[
Q_D \subset 
\frac{2 \pi \ii}{\lambda_1 - \lambda_2} \Z{}
\]
and in particular \(Q_D\) is generated by a single element, say 
\[
\varpi_0=\frac{2 \pi \ii}{\lambda_1 - \lambda_2} k_{12}
\]
for some integer \(k_{12}\).
By the same argument,
\[
\varpi_0=\frac{2 \pi \ii}{\lambda_a - \lambda_b} k_{ab}
\]
for some integer \(k_{ab}\), so that
\[
\lambda_a = \lambda_1 + \frac{2 \pi \ii}{\varpi_0} k_{a1}.
\]
Rescale to arrange that \(\varpi_0=1\), and let \(\lambda=\lambda_1\) so that \(\lambda_j = \lambda + 2 \pi \ii k_j\) for distinct integers \(k_j\).
Since now all of the \(\lambda_j\) lie on a vertical line in the complex plane, we can pick \(\lambda_1\) (by reordering these \(\lambda_j\)) to arrange that all \(k_j\) are positive.
If \(m\) is the greatest common divisor of the various \(k_j\), then \(1/m\) is a quasiperiod, so the \(k_j\) are relatively prime.
\end{proof}

\begin{lemma}\label{lemma:quasiperiods.quotients}
Every complex homogeneous surface  whose germ at a point is isomorphic to the germ of the action of \(G_D\) on \(\C{2}\) is a quotient by discrete subgroup \(\pi \subset Q_D \rtimes \C{}\).
\end{lemma}
\begin{proof}
Suppose that \(\delta \colon \C{2} \to \C{2}\) is a biholomorphism commuting with all elements of \(G_D\).
Write \(\delta\) as \(\pr{z,w} \mapsto \pr{Z,W}\), so that commuting with \(\pr{t,f} \in G_D\) is precisely
\begin{align*}
Z\of{z+t,w+f(z+t)} &= Z(z,w)+t, \\
W\of{z+t,w+f(z+t)} &= W(z,w)+f\of{Z(z,w)+t}.
\end{align*}
We can choose \(f\) to take on any value at any point.
So our first equation tells us that \(Z\of{z,w}\) cannot depend on \(w\), say \(Z(z+t)=Z(z)+t\), and setting \(z=0\) gives \(Z\of{z}=z+\varpi\) for some constant \(\varpi \in \C{}\).
Our second equation then says \(W(z+t,w+f(z+t))=W(z,w)+f\of{z+t+\varpi}\).
Take \(f=0\) to see that \(W(z,w)\) is independent of \(z\), say \(W(z,w)=W(w)\).
Differentiate in \(w\) to find that \(W'(w)\) is constant so that \(W(w)=s+\gamma w\), say, for some constants \(s, \gamma\).
Plug back in to find \(\gamma f(z) = f\of{z+\varpi}\) for all \(z\).
In particular, \(f(0)=0\) just when \(f\of{\varpi}=0\), for all \(f \in V_D\), so \(\varpi\) is a quasiperiod and \(\gamma=\gamma_{\varpi}\).
\end{proof}

We need to classify discrete subgroups \(\pi \subset Q_D \rtimes \C{}\), up to equivalence, i.e. up to conjugacy \(\delta \pi \delta^{-1}\) by automorphisms \pr{\delta,h} of \(\pr{X,G}=\pr{\C{2},G_D}\) fixing \(0 \in \C{2}\).
Lemma~\vref{lemma:aut.G.D} classifies the automorphisms.
Under the automorphism \(\delta(z,w)=\pr{z,\nu w}\), for some \(\nu \in \C{\times}\):
\(\pr{\varpi,s} \mapsto \pr{\varpi,\nu s}\).
Under the automorphism \(\delta(z,w)=\pr{z,w+f(z)}\) for some \(f \in V_{D+[0]}\): \(\pr{\varpi,s} \mapsto \pr{\varpi, s+f\of{\varpi}-\gamma_{\varpi} f(0)}\).
Write out the functions \(f \in V_{D+[0]}\) to see table~\vref{table:automorphisms}.

\begin{table}
\caption{The action of automorphisms of \(\pr{X,G}=\pr{\C{2},G_D}\) on \(Q_D \rtimes \C{}\). Here \(\left\{k_j\right\}\) is any finite set of relatively prime nonzero integers, \(\lambda \notin 2\pi \ii \Z{}\), \(\nu \in \C{\times}\) is arbitrary, and \(t \in \C{}\) is arbitrary.}%
\label{table:automorphisms}
\begin{tabular}{@{}AAA@{}}
\toprule
D & \text{Action} \\
\midrule
\sum\left[\lambda + 2 \pi \ii k_j\right] & \pr{k,s} \mapsto \pr{k,\nu s + t \pr{1-e^{\lambda k}}}\\
{[0]}+\sum\left[2 \pi \ii k_j\right] & \pr{k,s} \mapsto \pr{k,\nu s + tk}\\
\text{otherwise} & \pr{k,s} \mapsto \pr{k,\nu s} \\
\bottomrule 
\end{tabular}
\end{table}

We prove proposition~\vref{proposition:quotients.B.beta.1}.
\begin{proof}
For each divisor \(D\) we only have to check that the discrete subgroups \(\pi \subset Q_D \rtimes \C{}\) are precisely those occuring in our examples.
Suppose that \(\pi \subset \left\{0\right\} \rtimes \C{} \subset Q_D \rtimes \C{}\).
Then \(\pi\) is a discrete group of translations on \(\C{}\).
If \(D\) has positive degree at \(0\) then \(\pi \cap G=\pi\), so the induced effective group action is \(G'=G/\pi\) acting on itself by translations: \(B\beta{1}A0\)
If \(D\) has zero degree at \(0\), then \(G \cap \pi=\left\{0\right\}\), and the quotient group is \(G'=G_D\): \(B\beta{1}A1\).

So we can assume that \(\pi \subset Q_D \rtimes \C{}\) is not contained in \(\left\{0\right\} \rtimes \C{}\).
By lemma~\vref{lemma:Q.D}, we can assume that \(Q_D=\Z{}\) and that 
\[
D = [\lambda] + \sum_j \left[\lambda + 2 \pi \ii k_j\right]
\]
with relatively prime positive integers \(k_j\) and every \(f \in V_D\) satisfies \(f(z+1)=e^{\lambda} f(z)\).
The morphism \(Q_D \rtimes \C{} \to Q_D=\Z{}\) has image on \(\pi\) equal to \(n \Z{}\) for a unique integer \(n > 0\), so \(\pi\) contains an element \(\pr{n,s}\) for some \(s \in \C{}\).
The action of \(Q_D \rtimes \C{}\) on \(X=\C{2}\) is \((n,s) \in Q_D \rtimes \C{}, (z,w) \in \C{2} \mapsto \pr{z+n,e^{\lambda n}w+s}\); every nonidentity element \((n,s)\) acts without fixed points.
The exact sequence \(1 \to \pi_0 \to \pi \to n \Z{} \to 1\) defines a discrete subgroup \(\pi_0 \subset \C{}\).
Conjugation invariance of \(\pi\) under \(\pr{n,s}\) forces \(e^{\lambda n} \pi_0 =\pi_0\).
Under automorphisms of \(\pr{X,G}=\pr{\C{2},G_D}\), rescale \(Q_D \rtimes \C{}\) by \(\pr{n,s}\mapsto\pr{n,\nu s}\), so scale \(\pi_0 \subset \C{}\) by any \(\nu \in \C{\times}\).

If \(\pi_0=\left\{0\right\}\) then \(\pi\) is freely generated by a unique element \(\pr{n,s}\) with \(n > 0\).
If \(e^{\lambda n} \ne 1\) or if \(\lambda=0\) then use an automorphism of \(G\) from table~\vref{table:automorphisms} to arrange that \(s=0\): \(B\beta{1}B\).

If, after rescaling, \(\pi=\left<\pr{n,1}\right>\) with \(n > 0\) and \(e^{\lambda n}=1\), then say \(\lambda=2 \pi i m/n\) with \(m \in \Z{\times}\): \(B\beta{1}C\).

If \(\pi_0\) is generated by a single element, rescale to assume that that element is \(1\). 
To ensure \(e^{\lambda n}\pi_0=\pi_0\) we have \(e^{\lambda n}=\pm 1\), i.e. \(\lambda = \pi \ii m/n\) for some integer \(m\) and \(\pi=\left<\pr{n,s}, \pr{0,1}\right>\).
If \(\lambda=0\) then we have \(B\beta1D\).
Otherwise, if \(e^{\lambda n}=1\) then we have \(B\beta1E\).
If \(e^{\lambda n}=-1\), arrange by automorphism that \(s=0\), i.e. \(\pi=\left<\pr{n,0}, \pr{0,1}\right>\): \(B\beta1F\).

Suppose that \(\pi_0\) is generated by two elements, i.e. \(\pi_0 \subset \C{}\) is a lattice.
If \(\lambda=0\), then \(B\beta1G\), and if not but \(e^{\lambda n}=1\), then \(B\beta1H\).
So suppose that \(e^{\lambda n} \ne 1\); arrange by automorphism that \(s=0\) to find \(B\beta1I\).
\end{proof}

\section{%
\texorpdfstring%
{$B\beta2$: constant coefficient linear ODE with rescaling}
{Bbeta2: constant coefficient linear ODE with rescaling}
}%

\subsection{Definition}

Let \(\rG{D} = \C{} \times \C{\times} \times V_D\) with group operation
\[
\pr{t_0,\lambda_0,f_0(z)}
\pr{t_1,\lambda_1,f_1(z)}
=
\pr{t_0+t_1,\lambda_0 \lambda_1, 
f_0(z) + \lambda_0 \, f_1\of{z-t_0}}
\]
and action on \(\C{2}\)
\[
\pr{t,\lambda,f}(z,w)
=
\pr{z + t, \lambda w + f\of{z + t}}.
\]
If \(\deg D=1\), Lie denotes \(\rG{D}\) acting on \(\C{2}\) by the name \(C2\), and it is easy to see that the action is isomorphic to \(C2\) as described in section~\vref{section:C2}; henceforth we assume \(\deg D \ge 2\) and then Lie denotes \(\rG{D}\) acting on \(\C{2}\) by the name \(B\beta{2}\).

\subsection{Equivalent actions}

\begin{remark}\label{remark:GDprime.morphisms}
Consider three different families of morphisms.
\begin{enumerate}
\item
If \(g \in \rG{D}\), \(\delta(z,w)=g(z,w)\) is equivariant under the isomorphism \(h=\Ad{g} \colon \rG{D} \to \rG{D}\).
\item
Write an effective divisor on \(\C{}\) as \(D=\sum_j n_j \left[\lambda_j\right]\), where \(n_j \in \Z{}_{> 0}\) is a multiplicity and \(\lambda_j \in \C{}\) is a point.
Pick any nonzero complex numbers \(\mu, \nu\) and let \(E=\sum_j n_j \left[\mu \lambda_j\right]\).
Then \(\delta \colon \pr{z,w}\mapsto\pr{z/\mu,\nu w}\) is equivariant under
\[
h \colon (t,\lambda,f) \in \rG{D} \mapsto \pr{\frac{t}{\mu}, \lambda, \nu f\of{\mu z}} \in \rG{E}
\]
for any constants \(\mu, \nu \in \C{\times}\).
\item
Pick \(a \in \C{}\).
Let \(F=\sum_j n_j \left[\lambda_j+a\right]\).
The map \(\delta \colon \pr{z,w}\mapsto\pr{z,e^{az} w}\) is equivariant under the group morphism
\[
h \colon (t,\lambda,f) \in \rG{D} \mapsto \pr{t, e^{at} \lambda, e^{az} f(z)} \in \rG{F}.
\]
\end{enumerate}
Composing these isomorphisms together, we can carry out any complex affine transformation of \(D\).
\end{remark}

\begin{lemma}\label{lemma:aut.rGD}
Pick effective divisors \(D\) and \(E\) on \(\C{}\).
If a biholomorphism \(\delta \colon \C{2} \to \C{2}\) is equivariant for a group morphism \(h \colon \rG{D} \to \rG{E}\), then modulo a composition of the morphisms defined in remark~\ref{remark:GDprime.morphisms}, \(\delta=\id\) and \(E-D\) is effective or zero and \(h\) is the inclusion morphism.
The isomorphism classes of complex homogeneous spaces \(\pr{\C{2},\rG{D}}\) are parameterized by the isomorphism classes of effective divisors on \(\C{}\) modulo complex affine transformation.
\end{lemma}
\begin{proof}
Suppose that \(\delta \colon \C{2} \to \C{2}\) is a biholomorphism and \(h \colon \rG{D} \to \rG{E}\) is a Lie group homomorphism so that \(\delta \circ g = h(g) \circ \delta\) for any \(g \in \rG{D}\).
If \(h(g)=I\), then \(\delta \circ g=\delta\) so \(g=I\) and \(h\) is 1-1.
For any \(g \in \rG{E}\), we can replace \(\delta\) by \(g \delta\) and \(h\) by \(\Ad(g)h\).
Since \(\rG{E}\) acts transitively, we can assume without loss of generality that \(\delta(0,0)=(0,0)\).
Denote \(\delta(z,w)\) by \(\delta(z,w)=\pr{Z(z,w),W(z,w)}\) and \(h(t,\lambda,f)\) by \(h(t,\lambda,f)=\pr{T,\Lambda,F}\).
Equivariance under \(h\) says precisely that
\begin{align*}
Z\of{z+t,\lambda w + f(z+t)}
&=
Z\of{z,w}+T(t,\lambda,f), \\
W\of{z+t,\lambda w + f(z+t)}
&=
\Lambda(t,\lambda,f)W\of{z,w}+F_{t,\lambda,f}\of{Z(z,w)+T(t,\lambda,f)},
\end{align*}
for all \(\pr{t,\lambda,f}\in \rG{D}\) and all \(z,w \in \C{2}\).
In particular, if \(f=0\), we find
\begin{align*}
Z\of{z+t,\lambda w}
&=
Z\of{z,w}+T(t,\lambda,0), \\
W\of{z+t,\lambda w}
&=
\Lambda(t,\lambda,0)W\of{z,w}+F_{t,\lambda,0}\pr{Z(z,w)+T(t,\lambda,0)}.
\end{align*}
The Taylor series of the left hand side of the first equation depends on \(\lambda, w\) only through powers of \(\lambda w\), while the right hand side depends only on powers of \(\lambda\) and, separately, powers of \(w\), so both sides are independent of \(\lambda\) and \(w\): \(Z(z+t)=Z(z)+T(t,0,0)\).
So \(Z(z)=Z(0)+T(z)=T(z)\), and so \(Z\) is additive, and continuous, so \(Z(z)=c_1 z\) for some constant \(c_1 \in \C{\times}\) and \(T(t,\lambda,0)=c_1 t\).
Conjugate with the second type of morphism above to assume \(c_1=1\).
Plug back in to get \(T(t,\lambda,f)=t\).

Our equations are now
\[
W\of{z+t,\lambda w + f(z+t)}
=
\Lambda(t,\lambda,f)W\of{z,w}+F_{t,\lambda,f}\of{z+t}
\]
for all \(\pr{t,\lambda,f}\in \rG{D}\) and all \(z,w \in \C{2}\).
Again, expand the left hand side in \(\lambda w\), and then the right hand side can only involve \(\lambda\) terms and \(w\) terms in the same order, so \(\Lambda(t,\lambda,f)=\Lambda(t,f)\lambda^k\) and \(W(z,w)=W(z)w^k\) for some integer \(k\).
To have \(\delta\) a biholomorphism, \(k=1\), so \(\Lambda(t,\lambda,f)=\Lambda(t,f)\lambda\) and \(W(z,w)=W(z)w\).
Plug back in to get \(W(z+t)=\Lambda(t,f)W(z)\) and \(W(z)f(z)=F_{t,f}(z)\).
Clearly \(\Lambda(t,f)=\Lambda(t)\) is independent of \(f\), and a group homomorphism so \(\Lambda(t)=e^{at}\) and \(W(z)=ce^{at}\).
Using our second and third family of isomorphisms, arrange \(\Lambda(t)=1\) and \(W(z)=1\) and we find \(\delta=\text{id}\) and \(h=\text{id}\).
\end{proof}

\subsection{Quotient spaces}

\begin{lemma}
Suppose that \(D\) is an effective divisor on \(\C{}\).
Define an action \(\varpi \in Q_D, \pr{z,w} \in \C{2} \mapsto \pr{z+\varpi,\gamma_{\varpi} w} \in \C{2}\).
The biholomorphisms of \(X=\C{2}\) which commute with \(\rG{D}\) are precisely those arising in this action.
\end{lemma}
\begin{proof}
Because \(G_D \subset \rG{D}\), from lemma~\vref{lemma:quasiperiods.quotients} we know that any biholomorphism \(\pr{Z(z,w),W(z,w)}\) is \(Z(z,w)=z+\varpi, W(z,w)=\gamma_{\varpi}w+s\) for some quasiperiod \(\varpi\) of \(D\).
Equivariance under rescaling of \(w\) forces \(s=0\).
\end{proof}

Suppose that \(D\) is an effective divisor on \(\C{}\) of degree \(\ge 2\).
By the classification of quasiperiods in lemma~\vref{lemma:Q.D}, either \(Q_D=0\) or arrange that \(D=[0]+\sum_j \left[2 \pi i k_j\right]\) for relatively prime positive integers \(k_j\), and then \(Q_D=\Z{}\) and \(\gamma_{\varpi} = 1\) for all \(\varpi \in Q_D\).
The discrete subgroups \(\pi \subset Q_D\) are precisely the groups \(\pi=n\Z{}\), each for a unique nonnegative integer \(n \in \Z{}_{\ge 0}\).
The quotient spaces are therefore \(X'=\C{\times} \times \C{}\) with action of \(G'=\rG{D}/\left<\pr{n,1,0}\right>\).

\section{%
\texorpdfstring%
{$B\delta{3}, B\delta{4}$: the total space of a holomorphic line bundle on the projective line}%
{Bd3,Bd4: the total space of a holomorphic line bundle on the projective line}%
}\label{subsubsection:1.11}

\subsection{Definition}

Pick an integer \(n \ge 1\).
The surface \(X=\OO{n}\) is the total space of the usual line bundle \(\OO{n}=\OO{1}^{\otimes n} \to \Proj{1}\).
Each element of \(\OO{n}\) is a pair \((L,q)\) where \(L \subset \C{2}\) is a complex line through \(0\) and \(q \colon L \to \C{}\) is a homogeneous polynomial of degree \(n\).
The surface \(\OO{n}\) is acted on by the group \(\GL{2,\C{}}\) of linear substitutions of variables of \(\C{2}\), and is also acted on by the group \(\Sym{n}{\C{2}}^*\) by adding a globally defined homogeneous polynomial to the polynomial on any given line. 
The subgroup  \(\Z{}_n \subset \GL{2,\C{}}\) of scalings of variables by \(n\)-th roots of unity acts trivially. 
Consequently the group 
\[
G=\OnGroup
\]
acts on the surface \(X=\OO{n}\) with group operation
\[
\pr{g_0,p_0}\pr{g_1,p_1}
=
\pr{g_0g_1, p_0 + p_1 \circ g_0^{-1}}.
\]
Lie denotes the action of \(G\) on \(X\) by \(B\delta{4}\).

\subsection{The biholomorphism group}

\begin{lemma}\label{lemma:BiholomorphismGroupOn}
For any integer \(n \ge 0\),
\[
\Bihol{\OO{n}}
=
\begin{cases}
\Aff{\C{}} \rtimes \Hol{\C{}}{\PSL{2,\C{}}}, & n=0, \\
\OnGroup, & n>0
\end{cases}.
\]
\end{lemma}
\begin{proof}
For \(n=0\), clearly \(\OO{n}=\Proj{1} \times \C{}\), and the result is clear.
If \(n>0\), then by the Meyer--Vietoris sequence applied to affine charts
\[
\pr{Z,W}=\pr{\frac{1}{z},\frac{w}{z^n}},
\]
the homology groups are those of the zero section:
\[
\homology{*}{\OO{n},\Z{}}=\homology{*}{\Proj{1},\Z{}},
\]
but with intersection \([Z]^2=n\), where \(Z\) is the zero section. 
If \(f \colon \OO{n} \to \OO{n}\) is any biholomorphism, then the holomorphic curve \(Z'=f(Z)\) has homology class a generator of \(\homology{2}{\OO{n},\Z{}}\), so 
\(\left[Z'\right]=\pm \left[Z\right]\).
Deform \(Z\) through compact complex curves, to some other compact complex curve \(Z''\) in the same homology class which does not coincide with \(Z'\). 
Since \(Z'\) and \(Z''\) are compact complex curves,
\[
\left[Z'\right] \cap \left[Z\right] 
= \left[Z'\right] \cap \left[Z''\right] > 0.
\]
Therefore \(\left[Z'\right]=\left[Z\right]\).
The number of intersections of \(Z'\) with any fiber of \(\OO{n} \to \Proj{1}\) is just the degree of \(Z' \to \Proj{1}\), i.e. the integral
\[
\int_{Z'} \Omega,
\]
where \(\Omega\) is the pullback of the area form on \(\Proj{1}\) for the standard metric.
Clearly this depends only on homology class, so \(Z'\) intersects each fiber precisely once,
i.e. \(Z'\) is a holomorphic section of \(\OO{n} \to \Proj{1}\), and therefore is the
graph of a homogeneous polynomial in two variables. 
We can therefore compose our biholomorphism \(f\) with a biholomorphism from the group
\[
\OnGroup
\]
to arrange that \(f\) preserves the zero section, and that the composition  \AtoBtoCtoD{\Proj{1}}{0}{Z}{f}{Z'=Z}{}{\Proj{1}}
is the identity. 
Moreover, \(f\) preserves the subvarieties of \(\OO{n}\) which are the images of holomorphic sections of \(\OO{n} \to \Proj{1}\), i.e. the homogeneous polynomial functions  of two variables.

The tangent lines to the holomorphic sections of \(\OO{n} \to \Proj{1}\) are precisely the complex lines in the tangent planes of \(\OO{n}\) which are not tangent to the fibers of \(\OO{n} \to \Proj{1}\).
Therefore the tangent lines to the fibers of \(\OO{n} \to \Proj{1}\) are also preserved by \(f\), and so the fibers are preserved too.
We can write \(f\of{L,q}=\pr{L,r}\) where \(q, r \in \Sym{n}{L}^*\). 
For fixed
\(L\), this map is a holomorphic map of a complex line, fixing the origin, so a dilation \(f\of{L,q}=\pr{L,c(L)q}\).
This map \(c \colon \Proj{1} \to \C{}\) is holomorphic, so constant. 
Therefore \(f(L,q)=\pr{L,cq}\) and \(f\) lies in the group 
\[
\OnGroup.
\]
\end{proof}

\subsection{Quotient spaces}

\begin{lemma}
The action of \(G=\OnGroup\) on \(X=\OO{n}\) has no quotients.
\end{lemma}
\begin{proof}
For each point of \(X\), there is an element \(g \in G\) with that point as its unique fixed point.
Therefore every biholomorphism commuting with \(G\) fixes every point of \(X\): \(X\) has no quotients.
\end{proof}

Denote by \(\pm^n\) the group \(\pm 1\) if \(n\) is even, and \(1\) if \(n\) is odd.
Lie denotes the action of the subgroup 
\[
\SOnGroup \subset \OnGroup
\] 
on \(X=\OO{n}\) by \(B\delta{3}\).

\begin{lemma}
The action of \(G=\SOnGroup\) on \(X=\OO{n}\) has no quotients.
\end{lemma}
\begin{proof}
Every point of \(\Proj{1}\) is fixed by some element of \(G\).
Therefore every biholomorphism of \(X\) commuting with \(G\) fixes every point of \(\Proj{1}\).
Every biholomorphism of \(X\) is an element of \(\OnGroup\), say \(\pr{g,p}\).
Commuting with every element of \(G\), we must have \(\pr{g,p}=\pr{I,p}\).
The fixed points of \(\pr{I,p}\) are permuted by all elements of \(G\), and therefore there are no fixed points, i.e. the homogeneous polynomial \(p\) has no zeroes, a contradiction; so \(\pr{X,G}\) has no quotients.
\end{proof}

\section{%
\texorpdfstring%
{$B\gamma{4}$: restricting line bundles to $\C{}$}%
{Bgamma4: restricting line bundles to C2}%
}%
Pick a positive integer \(n\).
Let \(G^{\infty} \subset \GL{2,\C{}}/\Z{}_n\) be the subgroup fixing \(\infty \in \Proj{1}\), i.e. the group of \(2 \times 2\) complex matrices of the form
\[
\begin{pmatrix}
a & b \\
0 & d
\end{pmatrix}
\]
modulo roots of unity, i.e. modulo the matrices \(\lambda I\) where \(\lambda^n=1\) and \(\lambda \in \C{}\).
Consider the group
\[
G=G^{\infty} \rtimes \Sym{n}{\C{2}}^* 
\subset 
\OnGroup
\]
This subgroup preserves the domain of the affine chart, and acts transitively there.
The stabilizer of a certain point is the subgroup of pairs \((g,p) \in G^{\infty} \rtimes \Sym{n}{\C{2}}^*\) where
\[
g=
\begin{pmatrix}
a & 0 \\
0 & d
\end{pmatrix}
\]
with and \(p(0,1)=0\).
We let \(X \subset \OO{n}\) be the orbit of \(G\).

Let us be more explicit.
Take homogeneous coordinates 
\[
\begin{bmatrix}
Z_1 \\
Z_2
\end{bmatrix} \in \Proj{1}.
\]
Trivialize \(\OO{n}\) over the complement of \(\infty \in \Proj{1}\) by 
\[
\pr{z,w} \in \C{2} \mapsto \pr{L,q} \in \OO{n},
\]
where \(L\) is the line \(Z_1 = z Z_2\) and \(q=w \left. Z^n_2\right|_{L}\).
These \((z,w)\) are global coordinates on \(X\) so \(X=\C{2}\).
In these coordinates, elements 
\[
\pr{g,p} \in G
\] 
with 
\[
g = 
\begin{pmatrix}
a & b \\
c & d
\end{pmatrix}
\]
act on \(X\) by 
\[
\pr{g,0}\of{z,w}
=
\pr{\frac{az+b}{cz+d},\frac{w}{\left(cz+d\right)^n}}
\]
and if 
\[
p\of{Z_1,Z_2}
=
\sum_{i+j=n} c_{ij} Z_1^i Z_2^j,
\]
then
\[
(I,p)(z,w) = \pr{z,w+p\of{z,1}}
= \pr{z,w+\sum_{i+j=n} c_{ij}z^i}.
\]

\subsection{Quotient spaces}

The action of \(G\) on \(X=\C{2}\) includes all translations, so the biholomorphisms that commute with \(G\) are translation invariant, so are translations.
Only the translation by \(0\) commutes with all rescalings
\[
\pr{z,w} \mapsto \pr{\frac{z}{d},\frac{w}{d^n}}.
\]
Therefore there are no quotients of this action.

\section{%
\texorpdfstring%
{$B\gamma{1}$ and $B\gamma{2}$}%
{Bgamma1 and Bgamma2}%
}%

\subsection{Definition}

Consider the subgroup \(H^{\infty} \subset \GL{2,\C{}}/\Z{}_n\) of elements which fix the point \(\infty \in \Proj{1}\).
Let \(L_{\infty} \subset \C{2}\) be the line consisting of vectors
\[
\begin{pmatrix}
Z_1 \\
0
\end{pmatrix}
\in \C{2}.
\]
The elements \(g \in H^{\infty}\) act on the 1-dimensional vector space \(\Sym{n}{L_{\infty}}* \subset \C{2}\) with an eigenvalue, say \(\lambda_1\), and act on \(\Sym{n}{\C{2}/L_{\infty}}^*\) with another eigenvalue, say \(\lambda_2\).
Map \(g \in H^{\infty} \mapsto \pr{\lambda_1,\lambda_2} \in \C{\times} \times \C{\times}\); this map is onto with connected 1-dimensional kernel.
As in section~\vref{section:C8}, for each \(p \in \Proj{1}\) we have a connected complex subgroup \(H_p \subset \C{\times} \times \C{\times}\).
Each connected subgroup \(H_p \subset \C{\times} \times \C{\times}\) has preimage a connected subgroup \(H'_p \subset H^{\infty}\).
In turn this yields a connected subgroup \(G=G_p=H'_p \rtimes \Sym{n}{\C{2}}^*\) acting transitively on the affine plane \(X=\C{2} \subset \OO{n}\).
Writing out the group action, we find that if \(p=1 \in \Proj{1}\) then \(G_1=\rG{D}\) where \(D=n[0]\), so a special case of \(B\beta{2}\).

We want to connect to Lie's notation \cite{Lie:GA:5} p .767--773.
Pick a complex number \(c\) and an integer \(n \ge 1\). (Let \(k=n+1\) to match Olver's notation \cite{Olver:1995} p. 472, and \(r=n+1\) to match Lie's notation.)  
Consider the group of \(2 \times 2\) complex matrices of the form
\[
\mu 
\begin{pmatrix}
e^{\lambda \pr{1-\frac{c}{n}}} & b \\
0 & e^{-\lambda c/n}
\end{pmatrix}
\]
where \(\mu^n=1\), which is parameterized by arbitrary choices of the two complex numbers \(b\) and \(\lambda\).
This group contains the group of all rescalings by \(n\)-th roots of \(1\), so let \(C^n_c\) be the quotient by the \(n\)-th roots of \(1\).
Clearly \(C^n_c \subset \GL{2,\C{}}/\Z{}_n\), so we have a group
\[ 
C^n_c \rtimes \Sym{n}{\C{2}}^* \subset \OnGroup
\]
acting on \(\OO{n}\), fixing the line 
\[
\begin{bmatrix}
1 \\
0
\end{bmatrix} = \infty \in \Proj{1}.
\]
The group \(C^n_c \rtimes \Sym{n}{\C{2}}^*\) acts on \(\OO{n}\) preserving the affine plane \(X=\C{2} \subset \OO{n}\), the unique open orbit of this group. 
In our previous notation \(C^n_c=H_p\) where \(p=(c-n)/c\), so finite values of \(c\) correspond precisely to values of \(p\ne 1\).
The stabilizer of the origin of \(X\) is the group of \((g,p)\) of the form 
\[
g=
\mu
\begin{pmatrix}
e^{\lambda \pr{1-\frac{c}{n}}} & 0 \\
0 & e^{-\lambda c/n}
\end{pmatrix}
\]
where \(\mu^n=1\) and \(p\pr{0,1}=0\).
Any two different values of \(c\) give nonconjugate subroups \(C^n_c \subset \GL{2,\C{}}/\Z{}_n\).

\subsection{Quotient spaces}

The action of \(G\) on \(X=\C{2}\) includes all translations, so the biholomorphisms that commute with \(G\) must be translation invariant, so must be translations.
If \(c \ne 0\), (i.e. the group action in \(B{\gamma}1\)), only the translation by \(0\) commutes with all rescalings
\[
\pr{z,w} \mapsto \pr{e^{\lambda}z,e^{\lambda c}w}.
\]
Therefore there are no quotients of this action.
If \(c = 0\), (i.e. the group action in \(B{\gamma}2\)), only the translations by \(\partial_w\) commute with all rescalings
\[
\pr{z,w} \mapsto \pr{e^{\lambda}z,w}.
\]
The center of \(G\) consists precisely in these translations.
So the quotients are \(\pr{G',X'}=\pr{G/\pi,X/\pi}\) where \(\pi = \left\{0\right\} \times \Lambda\) for any discrete subgroup \(\Lambda \subset \C{}\).

\section{%
\texorpdfstring%
{$B\gamma{3}$}%
{Bgamma3}%
}%

\subsection{Definition}

Pick an integer \(n \ge 1\), which we will also write as \(n=k\) to match Olver's notation \cite{Olver:1995} p. 472.
Then consider the subgroup \(C'_n\) of elements 
\[
(g,p) \in \OnGroup
\]
of the form
\[
g=
\begin{pmatrix}
1 & b \\
0 & e^{-\lambda}
\end{pmatrix}
\]
and
\[
p\of{Z_1,Z_2}
=
\lambda Z_1^n
+
Z_2 r\of{Z_1,Z_2}
\]
with
\[
r\of{Z_1,Z_2} = 
\sum_{i+j=n-1} c_{ij} Z_1^i Z_2^j \in \Sym{n-1}{\C{2}}^*.
\]
Once again this subgroup acts preserving the domain of the affine plane \(X=\C{2} \subset \OO{n}\), where it acts transitively.
The stabilizer of a point is the subgroup \(b=0\) and \(r(0,1)=0\).

\subsection{Quotient spaces}
The action of \(G\) on \(X=\C{2}\) includes all translations, so the biholomorphisms that commute with \(G\) must be translation invariant, so must be translations.
No translations commute with all elements of \(G\) except translation by \(0\), so this action has no quotients.


\newpage
\section{The big table}

\newcommand{\lastaction}[5]
{
#1 & #2 & #3 
}

\newcommand{\action}[5]
{
\lastaction{#1}{#2}{#3}{#4}{#5}
\betweenEntries
}

\newcommand*{\sub}{\quad }

\begin{longtable}{@{}AAA@{}}
\caption[Complex homogeneous surfaces]{The classification of connected complex homogeneous surfaces, together with all faithful holomorphic actions of connected complex Lie groups on each such surface.
Each action is labelled according to Lie's table of holomorphic Lie algebra action germs on surfaces  \cite{Lie:GA:5} p .767--773 (also see \cite{Mostow:1950}, \cite{Olver:1995} p. 472).
The constants appearing in the table are constrained as follows: \(n\) is any integer \(\ge 1\), \(D\) is any effective divisor on \(\C{}\) with \(\deg D \ge 2\) (perhaps with some additional conditions detailed above), \(\alpha\) can be any complex number except that \(\alpha \ne 1\), \(\beta\) can be any complex number, \(\lambda\) can be any complex number with \(|\lambda|<1\), \(\Delta\) can be any discrete subgroup in either \(\C{}\) or \(\C{2}\) as required, and \(\Lambda\) is a lattice in \(\C{}\) or \(\C{2}\) as required.
Write \(\uAff{\C{}} \to G' \to \Aff{\C{}}\) to mean a connected complex Lie group which is a covering group of the group of complex affine transformations of \(\C{}\).
Let \(E_{\tau}=\C{}/\Z{}\left[1,\tau\right]\) and \(\omega=e^{\pi i /3}\).
Further information on this table is given in section~\vref{section:Lies.classification}.} \label{table:complexHomogeneousSurfaces}\\
\toprule
\textrm{Action} &
\textrm{Surface} &
\textrm{Group} 
\\
\midrule
\endfirsthead
\multicolumn{3}{l}%
{{\tablename\ \thetable{}: continued}} \\
\toprule
\textrm{Action} &
\textrm{Surface} &
\textrm{Group} 
\\
\midrule
\endhead
\bottomrule
\multicolumn{3}{r}{\text{continues on next page\ldots}} \\
\endfoot
\bottomrule
\endlastfoot
\action{A1}{\Proj{2}}{\PSL{3,\C{}}}
{\left[\begin{smallmatrix}
 a^0_0 & a^0_1 & a^0_2 \\
 0     & a^1_1 & a^1_2 \\
 0     & a^2_1 & a^2_2
\end{smallmatrix}\right]}{}
\action{A2}{\C{2}}{\GL{2,\C{}} \rtimes \C{2}}{\GL{2,\C{}}}{}
\action{A3}{\C{2}}{\SL{2,\C{}} \rtimes \C{2}}{\SL{2,\C{}}}{}
\action{B\beta{1}}{\C{2}}{G_D}{H_D}{\deg D \ge 2}
\action{\sub{B\beta{1}A0}}{\C{} \times \pr{\C{}/\Delta}}{G_D}{H_D}{\deg D \ge 2}
\action{\sub{B\beta{1}A1}}{\C{} \times \pr{\C{}/\Delta}}{G_D/\Delta}{H_D/\Delta}{\deg D \ge 2}
\action{\sub{B\beta{1}B0}}{\C{\times} \times \C{}}{G_D/\left<\pr{n,0}\right>}{H_D}{\deg D \ge 2}
\action{\sub{B\beta{1}B1}}{\C{\times} \times \C{}}{G_D}{H_D}{\deg D \ge 2}
\action{\sub{B\beta{1}C}}{\C{\times} \times \C{}}{G_D}{H_D}{\deg D \ge 2}
\action{\sub{B\beta{1}D}}{\C{\times} \times \C{\times}}{G_D/\left<\pr{n,0},\pr{0,1}\right>}{H_D}{\deg D \ge 2}
\action{\sub{B\beta{1}E}}{\C{\times} \times \C{\times}}{G_D/\left<\pr{n,s},\pr{0,1}\right>}{H_D}{\deg D \ge 2}
\action{\sub{B\beta{1}F}}{\C{\times} \to X' \to \C{\times}}{G_D}{H_D}{\deg D \ge 2}
\action{\sub{B\beta{1}G}}{\C{\times} \times \pr{\C{}/\Lambda}}{G_D}{H_D}{\deg D \ge 2}
\action{\sub{B\beta{1}H}}{\C{\times} \times \pr{\C{}/\Lambda}}{G_D}{H_D}{\deg D \ge 2}
\action{\sub{B\beta{1}I}}{\C{}/\Lambda \to X' \to \C{\times}}{G_D}{H_D}{\deg D \ge 2}
\action{B\beta{2}}{\C{2}}{\rG{D}}{\rH{D}}{\deg D \ge 2}
\action{\sub{B\beta{2}'}}{\C{\times} \times \C{}}{\rG{D}/\left<\pr{n,1,0}\right>}{\rH{D}/?}{\deg D \ge 2}
\action{B\gamma{1}}{\C{2}}%
{\Set{
e^{-a (n+\alpha)/n}
\twobytwo{e^a}{b}{0}{1}|a, b \in \C{}
}
\rtimes \Sym{n}{\C{2}}^*}%
{(g,p), b=0, p(1,0)=0}{\alpha \ne 1}
\action{B\gamma{2}}{\C{2}}%
{\Set{
e^{-a(n+1)/n}\twobytwo{e^a}{b}{0}{1}|a,b \in \C{}
}
\rtimes \Sym{n}{\C{2}}^*}%
{(g,p), b=0, p(1,0)=0}{}
\action{\sub{B\gamma{2}'}}{\C{} \times \pr{\C{}/\Delta}}%
{\Set{
e^{-a(n+1)/n}\twobytwo{e^a}{b}{0}{1}|a,b \in \C{}
}
\rtimes \pr{\Sym{n}{\C{2}}^*/\Delta}}%
{(g,p), b=0, p(1,0)=0}{}
\action{B\gamma{3}}{\C{2}}%
{
\Set{
\pr{
\twobytwo{1}{b}{0}{e^{-a}}
,Z_2 r\of{Z_1,Z_2}+a Z_1^n}
|
a,b \in \C{}, 
\deg r=n-1
}
}
{b=0, r(0,1)=0}{}
\action{B\gamma{4}}{\C{2}}%
{
\Set{
\twobytwo{*}{*}{0}{*}
/\Z{}_n
}
 \rtimes \Sym{n}{\C{2}}^*%
}
{(g,p), b=0, p(0,1)=0}{}
\action{B\delta{1}}{\C{2} \setminus 0}{\SL{2,\C{}}}%
{\twobytwo{1}{b}{0}{1}}{}
\action{\sub{B\delta{1}'}}{\pr{\C{2} \setminus 0}/z \sim \lambda z}{\SL{2,\C{}}}%
{
\twobytwo{1}{b}{0}{1}}{|\lambda|<1}
\action{B\delta{2}}{\C{2} \setminus 0}{\GL{2,\C{}}}%
{\twobytwo{1}{b}{0}{c}}%
{}
\action{\sub{B\delta{2}'}}{\pr{\C{2} \setminus 0}/z \sim \lambda z}{\GL{2,\C{}}/\left<\lambda I\right>}%
{\twobytwo{1}{b}{0}{c}}%
{|\lambda|<1}
\action{B\delta{3}}{\OO{n}}%
{\pr{\SL{2,\C{}}/\pm^n} \rtimes \Sym{n}{\C{2}}^*}%
{\left(
\twobytwo{a}{b}{0}{\frac{1}{a}}, p\right),
p(1,0)=1-\frac{1}{a^n}}{}
\action{B\delta{4}}{\OO{n}}%
{\OnGroup}%
{\left(
\twobytwo{a}{b}{0}{d}, p\right),
p(1,0)=1-\frac{1}{a^n}}{}
\action{C2}{\C{2}}{\C{} \times \Aff{\C{}}}{\br{0} \times \C{\times}}{}
\action{\sub{C2}'}{\pr{\C{}/\Delta} \times \C{}}{\pr{\C{}/\Delta} \times \Aff{\C{}}}{\br{0} \times \C{\times}}{}
\action{C3}{\C{2}}{\Aff{\C{}} \times \Aff{\C{}}}{\C{\times} \times \C{\times}}{}
\action{C5}{\Proj{1} \times \C{}}{\PSL{2,\C{}} \times \C{}}%
{\twobytwob{a}{b}{0}{a^{-1}}}{}
\action{\sub{C5}'}{\Proj{1} \times \pr{\C{}/\Delta}}{\PSL{2,\C{}} \times \pr{\C{}/\Delta}}%
{\twobytwob{a}{b}{0}{a^{-1}}}{}
\action{C6}{\Proj{1} \times \C{}}{\PSL{2,\C{}} \times \pr{\C{\times} \rtimes \C{}}}%
{\twobytwob{a}{b}{0}{a^{-1}} \times \C{\times}}{}
\action{C7}{\Proj{1} \times \Proj{1}}{\PSL{2,\C{}} \times \PSL{2,\C{}}}%
{\twobytwob{a}{b}{0}{\frac{1}{a}}\times\twobytwob{c}{d}{0}{\frac{1}{c}}}{}
\action{C8}{\C{2}}%
{\Set{\twobytwo{e^t}{0}{0}{e^{\alpha t}}| t \in \C{}} \rtimes \C{2}}%
{\twobytwo{e^t}{0}{0}{e^{\alpha t}}}%
{\alpha \ne 1}
\action{C9}{\Proj{1} \times \Proj{1} \setminus \Delta}{\PSL{2,\C{}}}%
{\twobytwob{a}{0}{0}{\frac{1}{a}}}{}
\action{\sub{C9}'}{\Proj{2} \setminus \pr{b^2=4ac}}{\PSL{2,\C{}}}%
{\twobytwob{a}{0}{0}{\frac{1}{a}},\twobytwob{0}{1}{1}{0}}{}
\action{D1}{\C{2}}{\C{2}}{0}{}
\action{\sub{D1_1}}{\C{\times} \times \C{}}{\C{\times} \times \C{}}{0}{}
\action{\sub{D1_2}}{\C{\times} \times \C{\times}}{\C{\times} \times \C{\times}}{0}{}
\action{\sub{D1_3}}{\pr{\C{}/\Delta} \times \C{}}{\pr{\C{}/\Delta} \times \C{}}{0}{}
\action{\sub{D1_4}}{\C{\times} \times \pr{\C{}/\Delta}}{\C{\times} \times \pr{\C{}/\Delta}}{0}{}
\action%
{\sub{D1_5}}%
{\C{\times} \to X' \to \C{}/\Delta}%
{\C{\times} \to G' \to \C{}/\Delta}%
{0}%
{}
\action{\sub{D1_6}}{\C{2}/\Lambda}{\C{2}/\Lambda}{0}{}
\action{D2}{\C{2}}{\uAff{\C{}}}{0}{}
\action{\sub{D2_1}}{\C{} \times \C{\times}}{\uAff{\C{}}}{(0,n)}{}
\action{\sub{D2_2}}{\C{} \times \pr{\C{}/\Lambda}}{\uAff{\C{}}}{(0,n+m\tau)}{E \text{ any elliptic curve}}
\action{\sub{D2_3}}{\C{\times} \times \C{}}{\uAff{\C{}}}{?}{}
\action{\sub{D2_4}}{\C{\times} \times \C{\times}}{\uAff{\C{}} \to G' \to \Aff{\C{}}}{?}{}
\action{\sub{D2_5}}{\C{\times} \times E_{\tau}}{\uAff{\C{}} \to G' \to \Aff{\C{}}}{?}{}
\action{\sub{D2_6}}{\C{\times} \times \C{}}{\uAff{\C{}} \to G' \to \Aff{\C{}}}{?}{}
\action{\sub{D2_7}}{\C{\times} \to X' \to \C{\times}}{\uAff{\C{}} \to G' \to \Aff{\C{}}}{?}{}
\action{\sub{D2_8}}{E_{\tau} \to X' \to \C{\times}}{\uAff{\C{}} \to G' \to \Aff{\C{}}}{?}{}
\action{\sub{D2_9}}{E_i \to X' \to \C{\times}}{\uAff{\C{}} \to G' \to \Aff{\C{}}}{?}{}
\action{\sub{D2_{10}}}{E_{\omega} \to X' \to \C{\times}}{\uAff{\C{}} \to G' \to \Aff{\C{}}}{?}{}
\action{\sub{D2_{11}}}{E_{\omega} \to X' \to \C{\times}}{\uAff{\C{}} \to G' \to \Aff{\C{}}}{?}{}
\action{\sub{D2_{12}}}{E_{\omega} \to X' \to \C{\times}}{\uAff{\C{}} \to G' \to \Aff{\C{}}}{?}{}
\action{\sub{D2_{13}}}{E_{\omega} \to X' \to \C{\times}}{\uAff{\C{}} \to G' \to \Aff{\C{}}}{?}{}
\action{\sub{D2_{14}}}{\pr{\C{}/\Lambda} \times \C{}}{\uAff{\C{}} \to G' \to \Aff{\C{}}}{?}{}
\lastaction{D3}{\C{2}}{\C{\times} \rtimes \C{2}}{\C{\times}}{}
\end{longtable}


\bibliographystyle{amsplain}
\bibliography{complex-homogeneous-surfaces}
\end{document}